\documentclass[twocolumn,journal]{IEEEtran}

\IEEEoverridecommandlockouts

\usepackage{cite}
\usepackage[english]{babel}
\usepackage{color}
\usepackage{amsmath}
\usepackage{amsfonts}
\usepackage{units}
\usepackage{amssymb}
\usepackage{mathptmx} 
\usepackage{times} 
\usepackage{graphicx}
\usepackage{epstopdf}
\let\labelindent\relax
\usepackage{enumitem}                                     

\newtheorem{theorem}{Theorem}
\newtheorem{lem}{Lemma}
\newtheorem{prop}{Proposition}

\newtheorem{cor}{Corollary}
\newtheorem{definition}{Definition}
\newtheorem{remark}{Remark}
\newtheorem{assumption}{Assumption}

\newcommand{\smallmat}[1]{\left[ \begin{smallmatrix}#1\end{smallmatrix} \right]}
\newcommand{\bigmat}[1]{\begin{bmatrix} #1 \end{bmatrix}}

\newtheorem{propty}{Property}

\newenvironment{proofof}{\noindent {\em Proof of }}{\hfill \hspace*{1pt}\hfill $\square$\\}

\begin{document}

%%%%%%%%%%%%%%%%%%%%%%%%%%%%%%%%%%%%%%%%%%%%%%%%%%%%%%%%%%%%%%%%%%%%%%%%%%%%%%%%%%%%%%%%%%%%%%%%%%%%%%%%%%%%%%%%%%%%%%%%%%%%%%%%%%
\title{Dynamic attitude planning for trajectory tracking in underactuated VTOL UAVs} % Title, preferably not more than 10 words.

\author{Davide~Invernizzi, Marco~Lovera\thanks{D. Invernizzi, M. Lovera are with Department of Aerospace Science and Technology, Politecnico di Milano, Via La Masa 34, 20156, Milano, Italy.} and Luca~Zaccarian\thanks{L. Zaccarian is with CNRS, LAAS, 7 avenue du colonel Roche, F-31400 Toulouse, France, Univ. de Toulouse, LAAS, F-31400 Toulouse, France, and Dipartimento di Ingegneria Industriale, University of Trento, Italy. Email to: davide.invernizzi@polimi.it}}

\maketitle

\begin{abstract}
This paper addresses the trajectory tracking control problem for underactuated VTOL UAVs. According to the different actuation mechanisms, the most common UAV platforms can  achieve only a partial decoupling of attitude and position tasks. Since position tracking is of utmost importance for applications involving aerial vehicles, we propose a control scheme in which position tracking is the primary objective.  To this end, this work introduces the concept of attitude planner, a dynamical system through which the desired attitude reference is processed to guarantee the satisfaction of the primary objective: the attitude tracking task is considered as a secondary objective which can be realized as long as the desired trajectory satisfies specific trackability conditions. Two numerical simulations are performed by applying the proposed control law to a hexacopter with and without tilted propellers, which accounts for unmodeled dynamics and external disturbances not included in the control design model. 
\end{abstract}


%%%%%%%%%%%%%%%%%%%%%%%%%%%%%%%%%%%%%%%%%%%%%%%%%%%%%%%%%%%%%%%%%%%%%%%%%%%%%%%%%%%%%%%%%%%%%%%%%%%%%%%%%%%%%%%%%%%%%%%%%%%%%%%%%%

\section{Introduction}
The increasing demand of complex and challenging applications involving Vertical-Take-Off and Landing Unmanned Aerial Vehicles (VTOL UAVs) has led to the design of novel configurations  to overcome the maneuverability limitation of standard platforms. In particular, multirotor UAVs with coplanar propellers cannot fulfill at the same time attitude and position tasks. We refer to these platforms as \emph{vectored-thrust} UAVs because their propulsive system can deliver a control force only along a fixed direction within the airframe. By designing an actuation mechanism that can change the thrust direction, a net force can be produced with respect to the aircraft frame, which allows to handle more complex maneuvering. \emph{Thrust-vectoring} UAVs can be realized by mounting the propellers in a non-coplanar fashion \cite{Romero2007,Jiang2014,Crowther2011,Rajappa2015} or by employing servo-actuators to dynamically adjust their orientation \cite{Hua2015,Ryll2015,Oosedo2015,InvernizziCCTA18}. While the dynamically tiltable configuration is efficient in terms of power consumption and may achieve full actuation \cite{Ryll2015}, it has a more complex mechanical structure. On the other hand, the fixed-tilted configuration is usually not fully actuated, as the control force can be delivered only in a compact region around the UAV vertical axis. The trajectory tracking control problem for the latter configuration is challenging because, due to underactuation, only a partial decoupling is possible between attitude and position objectives.

Much literature is devoted to the trajectory tracking problem for the vectored-thrust configuration \cite{Mahony} but few works address the thrust-vectoring one \cite{Franchi2016,Hua2015}. First of all, we show that it is impossible to track an arbitrary full pose trajectory (independent attitude and position) also for thrust-vectoring configurations, when the vectoring capability is limited, in particular, when the control force can be delivered only within a conic region around the vertical body axis. Therefore, we propose a control paradigm in which the realization of attitude tracking is subdued to the achievement of position tracking, which is the primary objective. Indeed, ensuring position tracking is mandatory in most application involving UAVs to guarantee safe operations. Our approach hinges on the development of a dynamic attitude planner, in the spirit of \cite{Hua2015,Franchi2016}. A dynamic attitude filter was proposed in \cite{Zhao2013} and applied to the vectored-thrust platform within a backstepping controller strategy,  with the purpose of avoiding the analytical computation of the reference angular velocity and acceleration. Instead, our design aims at prioritizing control objectives by properly modifying the desired attitude while guaranteeing compliance of the control law with the actuation constraints.  This is motivated by the fact that there always exists an attitude which guarantees that the platform can deliver, for the most common approximation of the actuation mechanisms, the control force required by position tracking. The control design is performed in two stages. The first step is the design of a control force and torque that ensures robust tracking of any desired trajectory that possesses certain smoothness and boundedness properties. In the second step, we show that attitude planning strategies proposed in existing works, \emph{e.g.},  \cite{Leeetal2010,InvernizziAUTO2018}, fit within the present robust control design framework. For what concerns the first step, this work takes inspiration from \cite{Roza2014}, in which the emphasis was on defining classes of position and attitude controllers that stabilize the UAV at a given position. Our contribution extends the results of \cite{Roza2014}, which dealt only with stabilization (constant reference) and vectored-thrust platforms and the results of \cite{InvernizziACC18}, which dealt only with a static attitude planning and vectored-thrust platforms. As done in previous contributions, we take advantage of the cascaded structure of the equations of motion for systems evolving on the manifold $\mathrm{SO}(3)\times \mathbb{R}^3$ and with the body-fixed frame coinciding with a principal axes frame. In particular, we tackle attitude tracking by selecting a control torque that does not cancel non-harmful nonlinearities, it has a simpler expression than the ones usually employed for UAV attitude control \cite{Leeetal2010}. This choice is a representative candidate for a large class of admissible attitude control laws that guarantee uniform asymptotic tracking (UAT), in the sense that it does not lead necessarily to an autonomous closed-loop system. Then, by designing a control force suitable to handle the different actuation constraints and that guarantees a (small signal) Input-to-State Stability (ISS) property for the perturbed position error dynamics, we study the stability of the cascade between the attitude and position closed-loop within the framework of differential inclusions. This approach simplifies the analysis of the corresponding cascaded non-autonomous system: the proof technique relies on casting the control problem as a stability problem for a compact attractor with dynamics satisfying regularity conditions that ensure robustness of the stability property against a very large class of (sufficiently small) perturbations \cite[Chapter 7]{Goebel2012}. Our proof is based on reduction theorems \cite{MaggioreZac17} which have been exploited with a different control strategy in \cite{Michieletto2017b} to address set-point tracking (stabilization) in multirotor UAVs. The control law guarantees robust tracking with semi-global properties and we show that by properly selecting the attitude planner, the control law satisfies the actuation constraints of both vectored-thrust and thrust-vectoring platforms.

The paper is organized as follows. The model for control design
is derived in Section \ref{sec:input_constr}  where the actuation limitations for the most common propulsive systems of VTOL UAVs are explained. Section \ref{sec:cont_probl} introduces the tracking control problem in $\mathrm{SO}(3)\times \mathbb{R}^3$ and in Section \ref{sec:cont_strat} the priority-oriented control paradigm is presented and the main result about the stability of the closed-loop system is stated. Section \ref{sec:att_plan} presents the development of the dynamic attitude planner. In Section \ref{sec:tilt_case}, two numerical examples are proposed
to test and verify the robustness of the control law against unmodelled dynamics and external disturbances on a simulation model of a hexacopter UAV with and without tilted propellers.

\begin{figure}
\begin{centering}
\includegraphics[scale=0.2]{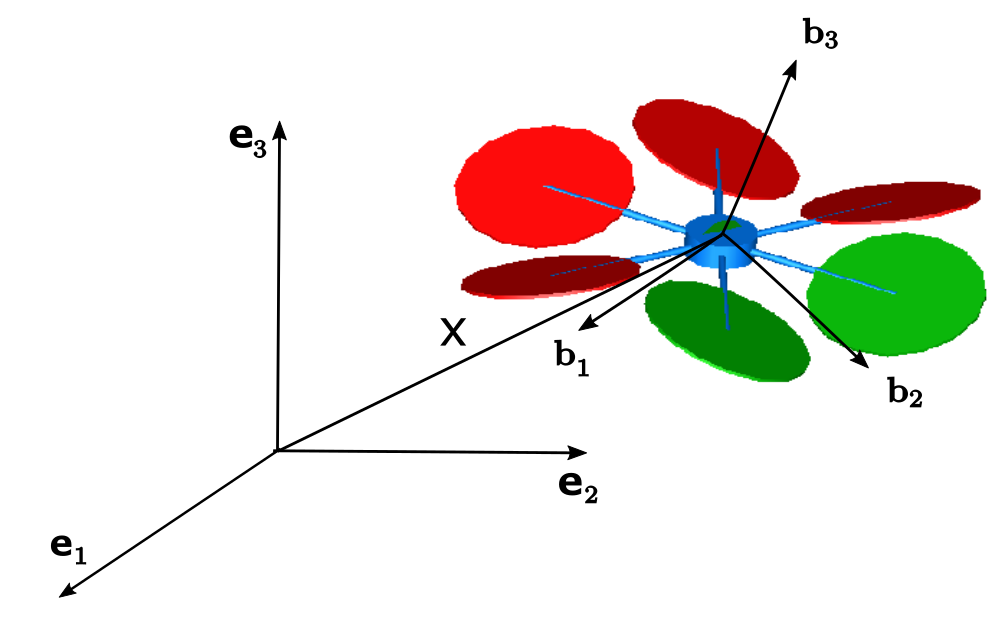}
\par\end{centering}
\caption{Reference frame definition.\label{fig:Reference-frame-definition}}
\end{figure}

\subsection{Notation}\label{sec:notation}
Throughout this paper, we use the following notation. For $x\in\mathbb{R}^n$, $\Vert x\Vert:=\sqrt{x_1^2+\ldots+x_n^2}$ is the corresponding Euclidean norm. For $A\in\mathbb{R}^{n\times n}$, $\Vert A\Vert_F:=\sqrt{\mbox{tr}(A^T A)}$ is the Frobenius norm, the minimum and maximum eigenvalues are denoted as $\lambda_m(A)$ and $\lambda_M(A)$, respectively, and $\mbox{skew}(A):=\frac{A-A^T}{2}$ is the skew-symmetric part of $A$. We use the compact notation $A\in \mathbb{R}^{n\times n}_{>0}$ to represent a positive definite matrix. 
The unit vectors corresponding to the canonical basis in $\mathbb{R}^n$ are $e_i:=[0,\ldots,\,1,\ldots,\,0]^T$ for $i=1,\,...,\,n$. The identity matrix in $\mathbb{R}^{n\times n}$ is denoted as $I_n:=[e_1,\cdots e_i\cdots,e_n]$. Given a symmetric matrix $K\in\mathbb{R}^{3\times 3}$, such that $\mbox{tr}(K)I_3-K\in\mathbb{R}^{3\times 3}_{>0}$,  and a rotation matrix $R\in\mathrm{SO}(3)$, we define the \emph{navigation} function as $\Psi_K(R):=\frac{1}{2}\mbox{tr}(K(I_3-R))$. If matrix $K$ has distinct eigenvalues, the navigation function has four non-degenerate critical points with one global minimum at $R=I_3$ \cite{koditschek1989}. The notation $R_u(\theta)$ is used to represent the rotation matrix corresponding to a rotation around a unit axis $u\in\mathbb{S}^2$ of an angle $\theta\in\mathbb{R}$. The normalized distance on $\mathrm{SO}(3)$, induced by the Frobenius norm in $\mathbb{R}^{3\times 3}$, is denoted as $\Psi(R):=\tfrac{1}{8}\Vert R-I\Vert_F=\sqrt{\frac{1}{4}\mbox{tr}(I_3-R)}$. By exploiting the angle-axis parametrization, $\Psi(R)$ can be written in terms of the rotation angle $\theta$ as $\Psi(R_u(\theta))=\sqrt{\frac{1-\cos(\theta)}{2}}$ \cite{koditschek1989}. Given a navigation function $\Psi_K$, the following inequality is valid:
\begin{equation}\label{eq:nav_funct}
\lambda_m(\mbox{tr}(K)I_3-K)\Psi^2(R)\leq\Psi_K(R)\leq \lambda_M(\mbox{tr}(K)I_3-K)\Psi^2(R).
\end{equation}
The set of piecewise-continuous and bounded functions is $L_\infty$. Given $f\in L_\infty$, we denote $f_m,\,f_M\in\mathbb{R}_{\geq 0}$ as lower and upper bounds, respectively, \emph{i.e.}, $f_m\leq \Vert f(t)\Vert \leq f_M,\,\forall t\geq 0$. The gradient of a real valued function $f:\mathbb{R}^n\rightarrow \mathbb{R}$ is denoted as $\nabla_x f(x)$.  Given the vectors $x,y$ we often denote $(x,y):=[x^T,y^T]^T$. Given $\omega\in\mathbb{R}^3$, the \emph{hat} map $\hat{\cdot}:\mathbb{R}^3\rightarrow \mathfrak{so}(3)$ is such that $\hat{\omega}y=\omega\times y,\;\forall y\in\mathbb{R}^3$. The inverse of the \emph{hat} map is the \emph{vee} map, denoted as  $(\cdot)^\vee:\,\mathfrak{so}(3)\rightarrow\mathbb{R}^{3}$, which is known to satisfy the useful property:
\begin{equation}\label{eq:veemap}
(R\hat{\omega}R^T)^\vee=R\omega \quad \mbox{for} \quad R\in\mathrm{SO}(3). 
\end{equation}

\section{Dynamical models for thrust-vectoring UAVs}
\label{sec:input_constr}

This section recalls the equations of motion of a rigid body moving in a constant gravity field $-ge_3,\,g\in\mathbb{R}_{>0}$, and actuated by a control wrench $w_c:=(f_c,\tau_c)$, where $f_c$ and $\tau_c$ are the control force and torque, respectively. These assumptions may be employed to design control laws for a large class of small-scale UAVs in which the components are sufficiently rigid, the flight conditions are such that the aerodynamics effects can be dominated with high gain control and the actuators dynamics is fast enough. The control wrench is generated by means of different kind of propulsive systems.  In order to obtain a control design independent from the specific actuation mechanism, we will assume the control wrench as the input variable. Then, we will resort to approximate models of the wrench map (from the physical inputs to the delivered wrench), according to most common actuation mechanisms. In particular, the limitations in terms of thrust-vectoring capabilities are formally defined for the most common UAV configurations.

The motion of a rigid body can be described by the one of a body-fixed frame $\mathcal{F}_B=(O_B,\{b_1,b_2,b_3\})$ with respect to an inertial frame $\mathcal{F}_I=(0_I,\{e_1,e_2,e_3\})$, as shown in Figure~\ref{fig:Reference-frame-definition} (for the sake of simplicity we assume that the inertial frame triad coincides with the standard basis of $\mathbb{R}^3$).
Under the approximation introduced at the beginning of this section, the equations of motion for control design are \cite{Mahony}:
\begin{align}\dot{x} & =v \label{dyn4cont1}\\
\dot{R} & =R\hat{\omega}\label{dyn4cont2}\\
m\dot{v} & =-mge_{3}+Rf_{c} \label{dyn4cont3}\\
J\dot{\omega} & =-\hat{\omega} J\omega+\tau_{c},\label{dyn4cont4}
\end{align}
where $x\in\mathbb{R}^3$ and $v\in\mathbb{R}^3$ are position and velocity of the center of mass in the inertial frame, $R\in\mathrm{SO}(3)$ and $\omega\in\mathbb{R}^3$ are the rotation matrix and the body angular velocity, while $m\in\mathbb{R}_{>0}$ and $J=J^T\in\mathbb{R}^{3\times 3}$ are the mass and inertia matrix with respect to the principal axes of the rigid body. According to our choice, the translational motion evolves in the inertial frame whereas the rotational motion in the body one. 

When the range of the map from the set of physical inputs $U$ to the delivered wrench, namely $U \ni u\mapsto w_c(u)$,  spans $\mathbb{R}^6$, the system is fully actuated. In the following, the control torque $\tau_c$ is assumed to span $\mathbb{R}^{3}$, whilst the actuation mechanism allows to deliver the control force $f_c$ only on a compact subset of $\mathbb{R}^3$. As a consequence, the tracking control problem, that we are going to present in the next Section, becomes more challenging. Within this category of UAVs, a further classification distinguishes between \emph{vectored-thrust} \cite{InvernizziACC18} and \emph{thrust-vectoring} configurations \cite{Franchi2016}.
\begin{enumerate}[label=(\alph*)]
\item The first class includes vehicles with a \emph{vectored-thrust} configuration, typical of multirotor UAVs with coplanar propellers. In this case (Figure \ref{fig:act_const}-a), the control force can be delivered only in the common direction of the propellers axis. Hence, the control force $f_c:\mathbb{R}_{\geq 0}\rightarrow \mathbb{R}^3$ has to satisfy the following constraint:
\begin{align}
f_{c_3}(t)>0,\quad
f_{c_1}(t)=f_{c_2}(t)=0\quad\forall t\geq 0.\label{eq:coplanar_const}
\end{align}
This configuration is adopted in most multirotor UAVs thanks to the inherent mechanical simplicity combined with good performance in many flight conditions \cite{Mahony}. This configuration allows only for arbitrary rotational motion around the thrust axis.
\item The second class comprises UAVs with \emph{thrust-vectoring} capabilities. Multirotor UAVs both with fixed-tilted \cite{Rajappa2015} or dynamically tiltable propellers \cite{Ryll2015}, have been proposed in recent years to overcome the intrinsic maneuverability limitation of the coplanar platform. While the dynamically tiltable configuration makes the system fully actuated if the tilt angle of the servo-actuators is not limited, this is not true for the fixed-tilted configuration, for which the maximum inclination at which the rotors are mounted is limited by the efficiency loss: the power consumption in hover is large and increases with the inclination of the propellers.  In this case  (Figure \ref{fig:act_const}-b), the control force $f_{c}$ can span, approximately, the conic region defined around the third body axis:
\begin{align}
&0<\cos(\theta_{M})\leq\frac{f_{c}(t)^{T}e_{3}}{\Vert f_{c}(t)\Vert }=\cos(\theta_c(t))\quad\forall t\geq 0,\label{eq:tilt_const}
\end{align}
where $\theta_M$ is the maximum tilt angle.   
%\item The case of multirotor UAVs with orthogonally mounted and fixed propellers is represented in Figure \ref{fig:act_const}-c. In this condition, the actuation limitation may be approximated by the \emph{cylindrical} region defined by:
% \begin{align}
%f_{c_3}(t)>0,\quad
%f_{c_1}^2(t)+f_{c_2}^2(t)\leq f_{p_M}^2\quad\forall t\geq 0,\label{eq:octocop_const}
%\end{align}
%in which the maximum lateral force is given by $f_{p_M}\in \mathbb{R}_{>0}$. The \emph{octocopter} with four small rotors which are orthogonally mounted with respect to the four main rotors is representative of this case \cite{Romero2007}.
\end{enumerate}
Finally, for all the cases above, we assume that the control force is bounded by a positive scalar $T_M\in\mathbb{R}_{>0}$:
\begin{equation}
\Vert f_c(t)\Vert\leq T_M \quad \forall t\geq 0.
\end{equation}
This constraint is required to account for actuators saturations, in particular the limited spinning velocity of propellers.
\begin{figure}
\begin{centering}
\includegraphics[scale=0.7]{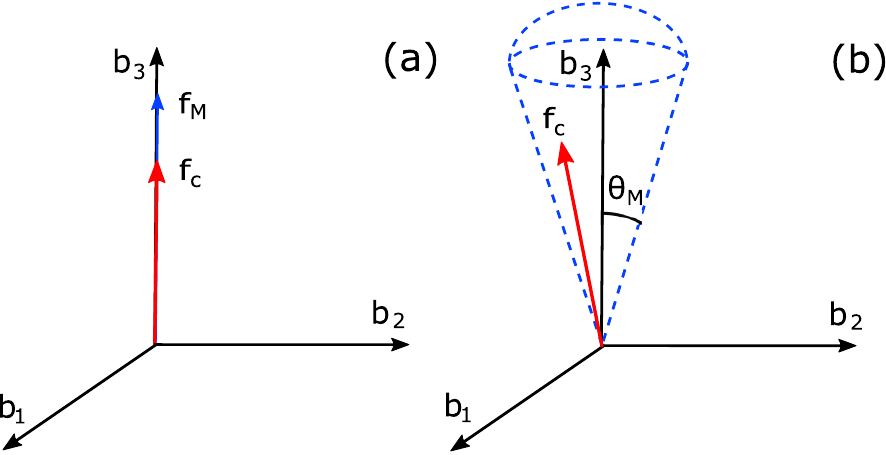}
\par\end{centering}
\caption{Vectored-thrust (a) - thrust-vectoring (b) configuration. \label{fig:act_const}}
\end{figure}

\section{Control problem: desired trajectory and steady state inputs}\label{sec:cont_probl}
This work deals with the tracking control problem for the system described by \eqref{dyn4cont1}-\eqref{dyn4cont4}. The objective is to track a reference $t\mapsto (R_{d}(t),\omega_d(t),x_{d}(t),v_d(t)) \in T\mathrm{SO}(3)\times\mathbb{R}^6$, where $T\mathrm{SO}(3):=\mathrm{SO}(3)\times \mathbb{R}^3$ is the trivial tangent bundle of $\mathrm{SO}(3)$.

When considering the actuation limitations shown in the previous section, the tracking of an arbitrary reference trajectory in $T\mathrm{SO}(3)\times\mathbb{R}^6$ is not possible due to the platform underactuation. For instance, it is well known that the standard quadrotor cannot hover in non-level attitude and position tracking can be achieved only by properly modifying the attitude. However, platforms with thrust-vectoring capabilities can achieve a certain degree of decoupling between attitude and position tracking when the corresponding trajectories are compatible in some sense. To suitably represent and exploit these degrees of freedom, let us first compute the steady state (short, $ss$) control wrench, which is obtained by inverting the system dynamics: 
\begin{align}
f_{c}^{ss}(t)	&:=mR_{d}^{T}(t)\left(\dot{v}_{d}(t)+ge_{3}\right)\label{eq:feed_force}\\
\tau_{c}^{ss}(t) &:=	J\dot{\omega}_{d}(t)+\hat{\omega}_{d}(t)J\omega_{d}(t).\label{eq:feed_torque}
\end{align}
For arbitrary position and attitude trajectories, it is likely that $f_{c}^{ss}$ will not be compatible with the actuation limitations \eqref{eq:coplanar_const} and \eqref{eq:tilt_const}. On the other hand, because position tracking is mandatory in aerial applications, equation \eqref{eq:feed_force} suggests that the desired attitude can be properly modified to be compliant with the actuation constraints. Indeed, according to equation \eqref{eq:feed_force}, the control force is obtained by rotating the vector $m\left(\dot{v}_{d}+ge_{3}\right)$ by $R_d^T$. The  rationale behind the proposed control is to prioritize position over orientation tracking, as already suggested by \cite{Hua2015} and \cite{Franchi2016}. Following this path, the attitude motion cannot be assigned arbitrarily anymore. However, we will propose a strategy to track the desired attitude as long as it allows to follow the desired position. Whenever this condition cannot be granted, only the closest feasible attitude will be tracked. In particular, we will exploit the fact that when the control force is delivered along the vertical direction of the body frame ($b_3$) (coplanar case), the constraint \eqref{eq:tilt_const} for the thrust-vectoring case is satisfied as well.

Before going on, the following assumptions are required to hold for the desired trajectory.
It is worth remarking that these conditions are not too restrictive for standard applications involving multirotor UAVs. 

\begin{assumption}\label{assum_setpoint} \emph{Smoothness and boundedness of the desired trajectory.} The desired trajectory $t\mapsto (R_{d}(t),\omega_d(t),x_{d}(t),$ $v_d(t)) \in T\mathrm{SO}(3)\times\mathbb{R}^6$ satisfies $\dot{R}_d(t)=R_d(t)\hat{\omega}_d(t)$ and $\dot{x}_d(t)=v_d(t)$, for all $t\geq 0$. Moreover,
\begin{enumerate}
\item the desired trajectory $x_d(\cdot)$ belongs at least to ${C}^4$;
\item the desired acceleration $\dot{v}_d(\cdot)$ and the desired jerk $\ddot{v}_d(\cdot)$ belong to $L_\infty$; in particular, the acceleration profile is such that the nominal force \eqref{eq:feed_force} is bounded by some strictly positive constants $f_m^{ss},\,f_M^{ss}$:
\begin{equation}
0<f^{ss}_m \leq \Vert f^{ss}_c(t)\Vert\leq f^{ss}_M<T_M\quad \forall t\geq 0; \label{ssbounds}
\end{equation}
and $\inf_{t\geq 0}(m\vert g+\dot{v}_{d_3}\vert)>0$
\item the desired angular velocity is bounded and continuously differentiable, \emph{i.e.}, $\omega_d(\cdot)\in {C}^1\cap L_\infty$.
\end{enumerate}
\end{assumption}

\section{Robust stabilization of the error dynamics}
\label{sec:cont_strat}

In this section we show the design of a control law that ensures the position tracking for any bounded and sufficiently smooth trajectory with no restriction on the initial position error and some restriction on the initial attitude error. Due to the actuation limitations, full decoupling of the position and attitude tracking cannot be achieved, as discussed in the previous section. Specifically, we will show that the system dynamics can be represented by a feedback interconnection in which an attitude planner provides the reference to the attitude subsystem which, in turn, affects the position error dynamics (see Figure \ref{fig:cont_scheme}). 

The attitude planner design plays a central role in our control strategy. This subsystem takes the desired attitude reference and the position errors as inputs and  is in charge of computing a reference attitude and angular velocity that are feasible, in the sense that the actuation constraints are satisfied.  The output of the planner (which is the actual reference to the attitude subsystem) is denoted  as $(R_p,\omega_p)\in T\mathrm{SO}(3)$.

\begin{propty}\label{propty:planner_ref}
The attitude planner motion $\left(R_p,\,\omega_p\right) \in T\mathrm{SO}(3)$ is feasible, in the following sense % that the following conditions hold $\forall t\geq 0$:
\begin{align}\label{eq:planner_setpoint}
\dot{R}_p(t)&=R_p(t)\hat{\omega}_p(t), \quad \forall t\geq 0\\
\omega_{p}(\cdot)&\in {C}^1\cap L_\infty.
\end{align}
\end{propty}

\begin{figure}
	\begin{centering}
		\includegraphics[scale=0.38]{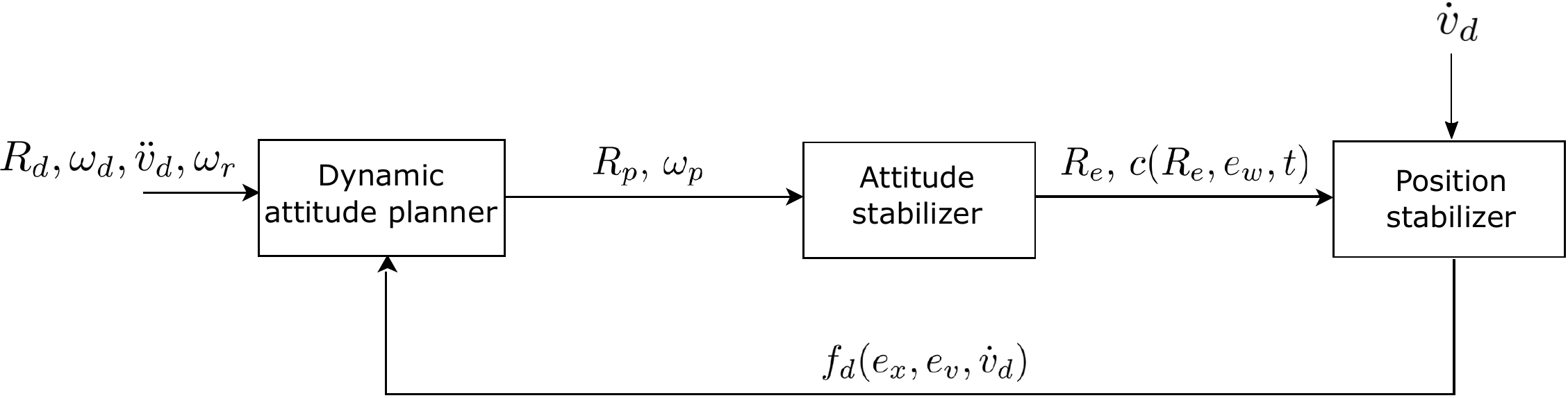}
		\par\end{centering}
	\caption{Overall control scheme comprising the dynamic attitude planner and the error stabilizers.} 
	\label{fig:cont_scheme}
\end{figure}

\subsection{Control force and position error dynamics}
Since the goal of this work is to obtain a control law that prioritizes position over orientation tracking, we start by inspecting the position error dynamics. By focusing on the  dynamics in \eqref{dyn4cont1}-\eqref{dyn4cont4}, the natural choice of the tracking errors for position and velocity is:
\begin{align}
e_{x} := x-x_{d},\quad e_{v} := v-v_{d}.\label{eq:pos_err}
\end{align}
Accordingly, the position and velocity errors are expressed in the inertial frame. The control objective is to stabilize the origin $(e_x,e_v)=(0,0) \in \mathbb{R}^6$.
%  defined by
%\begin{equation}
%{\mathcal{A}_p}:=\left\{ \left(e_{x},\,e_{v}\right)\in{\mathbb{R}^6}:\,e_{x}=0,\,e_{v}=0\right\}.\label{eq:pos_attrac}
%\end{equation}
Consider the equations \eqref{dyn4cont1}-\eqref{dyn4cont3} and the definition of the errors in \eqref{eq:pos_err}. Then,
\begin{align}
\dot{e}_{x}	&=e_{v}\label{pos_err_dyn1}\\
m\dot{e}_v&=m\left(\dot{v}-\dot{v}_d\right)=-m\left(\dot{v}_d+g e_3\right)+Rf_c.
\end{align}
The system would be fully actuated if one could arbitrarily assign $f_c$, and it would be possible to enforce $Rf_c=f_d$, with 
\begin{equation}
f_d (e_x,e_v, \dot{v}_{d}):= \beta\left(e_{x},\, e_{v}\right)+m\left(\dot{v}_{d}+ge_{3}\right),\label{eq:des_force}
\end{equation}
where $\beta:\mathbb{R}^6\rightarrow\mathbb{R}^{3}$ is a static state feedback stabilizer. The corresponding closed-loop dynamics would be described by \eqref{pos_err_dyn1} and 
$m\dot{e}_{v}=\beta\left(e_{x},\, e_{v}\right)$. Thus, the origin
%set $\mathcal{A}_p$ defined in \eqref{eq:pos_attrac} 
would be globally asymptotically stable. 
In this work we will adopt the same strategy used in \cite{Naldi2017}, corresponding to a specific version of nested saturations-based static nonlinear feedback:
\begin{equation}
\label{eq:betadef}
  \beta\left(e_{x},\, e_{v}\right):= \lambda_2 \mbox{sat}\left(\frac{k_2}{\lambda_2} \left( e_v +\lambda_1\mbox{sat} \left(\frac{k_1}{\lambda_1} e_x \right)\right)\right),
\end{equation}
where $k_1,k_2$ are stabilizing gains and $\lambda_1,\lambda_2$ are suitably chosen saturation levels (see \cite{Naldi2017} for the details). While we develop our theory for this specific stabilizer, we emphasize that similar generalizations to those reported in \cite{Roza2014} are possible, thus allowing for stabilizing laws inducing improved transients.

% Up to now, any stabilizer that ensures global asymptotic stability of the double integrator system can be employed.
To account for the fact that it is not possible to obtain $Rf_c=f_d$ when the control force cannot span $\mathbb{R}^3$, it is convenient to rewrite the velocity error dynamics as:
\begin{align}
m\dot{e}_{v}&=-m\left(\dot{v}_d+g e_3\right)+RR_p^T R_pf_c\\
&=-m\left(\dot{v}_d+g e_3\right)+R_e R_pf_c.\label{pos_err_dyn2_1}
\end{align}
where $R_p$ is the reference attitude given by the attitude planner and
\begin{equation}
R_e:=RR_p^T\in\mathrm{SO}(3)\label{eq:att_err}
\end{equation}
is the so-called left attitude error \cite{Bullo1999}. Introducing the corresponding angular velocity error,
\begin{equation}
e_\omega:=\omega-\omega_p,\label{eq:ang_vel_err}
\end{equation}
a natural choice for the control force $f_c$ is
\begin{equation}
f_c:=R_p^T\Phi\left(R_e,\,e_\omega,t\right)f_d,\label{eq:cont_force}
\end{equation}
where $\Phi(R_e,\,e_\omega,t):T\mathrm{SO}(3)\times \mathbb{R}_{\geq 0}\rightarrow \mathbb{R}^{3\times 3}$ is a design function, selected later, ensuring that $\Phi(R_e,e_\omega,t)\rightarrow I_3$ for $(R_e,e_\omega)\rightarrow (I_3,0)$ and $\forall t\geq 0$. Indeed, by adding and subtracting $f_d$ in equation \eqref{pos_err_dyn2_1}, the closed-loop velocity error dynamics reads:
\begin{equation}
m\dot{e}_{v}	=\beta\left(e_{x},\, e_{v}\right)+\Delta  R\left(R_{e},e_{\omega},t\right)f_d (e_x,e_v, \dot{v}_{d})\label{pos_err_dyn2}\\
\end{equation}
where
\begin{equation}
\Delta R(R_e,\,e_\omega,t):= R_e \Phi\left(R_e,\,e_\omega,t\right)-I_3.\label{eq:att_mismatch}
\end{equation} 
Written in this form, it is clear that the position error is clearly affected by the attitude error through the term $\Delta R f_d$, which is the mismatch between the desired force $f_d$ and the actual control force resolved in the inertial frame, \emph{i.e.}, $R f_c$. 

\begin{remark}\label{remark:cont_force}
	The control force is obtained by scaling the desired force $f_d$ by a suitable term $\Phi\left(R_e,\,e_\omega,t\right)$ dependent on the attitude error, and then by applying a rotation $R_p^T$ to the resulting vector.
	The idea behind the proposed control law is twofold. When the attitude error is large, the scaling term $\Phi\left(R_e,\,e_\omega,t\right)$ can be used to reduce the control force and, correspondingly, the overshoot in the position tracking is limited and the transient behavior can be improved. The second point is related to the fact that the desired attitude may be such that the constraints on $f_c$ expressed by equations \eqref{eq:coplanar_const}-\eqref{eq:tilt_const} are still not respected with the control force defined in equation \eqref{eq:cont_force}. However, by prioritizing position over orientation tracking one can consider the planner rotation matrix $R_p$ as an additional degree of freedom to ensure that the scaled control force $\Phi\left(R_e,\,e_\omega,t\right)f_d$ is  rotated to be compliant with the actuation constraint.  
\end{remark}

The rationale behind the proposed control law is that if the attitude error dynamics can be made asymptotically stable, then, for $t\rightarrow\infty$, $R_{e}(t)\rightarrow I_3$, $\Phi\left(R_e(t),\,e_\omega(t),t\right)\rightarrow I_3$ and the control force delivered in the inertial frame $R(t)f_{c}(t)=R_e(t)\Phi(t) f_d(t)\rightarrow f_d(t)$, which is the force in the inertial frame required to track the desired position trajectory. Indeed, in this case, the mismatch between the desired and actual control force converges to zero as well, namely $\Delta R (t)f_d(t)\rightarrow 0$.
However, to prove this idea, we have to study the stability of the complete system, including the attitude error dynamics. In particular, we have to make assumptions on the attitude error dependent matrix $\Phi$ so that the position error dynamics has certain desired properties.
First of all, the choice of $\Phi$ must be such that the following property holds true for the attitude mismatch $\Delta R$. 
\begin{propty}\label{propty:van_pert}
\emph{(Vanishing perturbations)}. Consider $\Delta  R$ defined in equation \eqref{eq:att_mismatch}. Given $V_a(R_e,e_\omega):=\sqrt{\Vert e_\omega \Vert^2+\Psi^2(R_e)}$,  where $\Psi(R_e)$ is defined in Section \ref{sec:notation}, there exists a bounded class-$K$  function $\gamma:\mathbb{R}_{\geq 0}\rightarrow \mathbb{R}_{\geq 0}$, satisfying $\forall \left(R_e,\,e_\omega,t\right)\in T\mathrm{SO}(3)\times\mathbb{R}_{\geq 0}$: 
\begin{equation}
\Vert \Delta  R(R_e,e_\omega,t)\Vert \leq \gamma\left(V_a\left(R_e,\,e_\omega\right)\right).
\label{eq:vanishing}
\end{equation}
\end{propty}

We will now show a possible selection of $\Phi$ such that the  above property is verified. The expression  proposed below is only one possible choice that we propose as a starting one. A wide range of performance-oriented alternative choices are possible as long as they satisfy Property \ref{propty:van_pert}, which is needed to ensure that, at convergence, the magnitude of the delivered control force $f_c$ converges to the nominal force $f_{c}^{ss}$ \eqref{eq:feed_force}. It is straightforward to employ a scaling transformation, dependent on the attitude error alone, as follows:
\begin{equation}
\Phi(R_e,\,e_\omega,t):=c(R_e,t)I_3, \label{eq:scaling_fact}
\end{equation}
where $c:\mathrm{SO}(3)\times \mathbb{R}_{\geq 0}\rightarrow \mathbb{R}$ is a properly selected function.
 The next proposition, whose proof is given in Appendix \ref{proof_scaling_funct},
gives an example of such scaling function, which is naturally written in terms of the angle $\theta_e$ between the desired direction $b_{p_3}:=R_p e_3$ and the vertical body axis $b_3:=R e_3$:
\begin{equation}\label{eq:theta_e}
\theta_e(R_e,R_p):=\arccos(b_{p_3}^T b_3 )=\arccos(e_3^T R_p^T R_e R_p e_3).
\end{equation}

\begin{prop}\label{prop:scaling_funct}
	Given $c(R_e,t):=\frac{\ell-(1-\cos(\theta_e(R_e,R_p(t))))}{\ell}$, where $\theta_e$ is defined in \eqref{eq:theta_e}, then, for $\Phi(R_e,e_\omega,t):=c(R_e,t)I_3$ and  $\ell>2$, Property \ref{propty:van_pert} is satisfied.
\end{prop}
% \begin{proof}
% See Appendix \ref{proof_scaling_funct}. 
% \end{proof}

\subsection{Control torque and attitude error dynamics}

The attitude controller has to ensure the convergence of the attitude tracking errors according to the fully actuated rotational dynamics in equations \eqref{dyn4cont2}, \eqref{dyn4cont4}. In this work, we will avoid the use parametrization and we will resort to a geometric approach. In particular, by using matrix multiplication as the group operation, the  attitude error $R_e=RR_p^T$, is employed as the attitude error measure in $\mathrm{SO}(3)$, which was already introduced in equation \eqref{eq:att_err}. The error kinematics can be derived from \eqref{dyn4cont2} and \eqref{eq:planner_setpoint} as  follows:
\begin{equation}\label{eq:Re_dot}
\dot{R}_{e}=\dot{R}R_p^T+R\dot{R}_p^{T}=R_{e}R_p\hat{e}_{\omega}R_p^T.
\end{equation}
Consider the system defined in \eqref{dyn4cont2}, \eqref{dyn4cont4} and the control law
\begin{equation}
\tau_{c}  := -R_p^Te_{R}-K_{\omega}e_{\omega}+J\dot{\omega}_p+ \hat{\omega}_p J\omega,\label{eq:cont_torque}
\end{equation}
where $K_\omega\in\mathbb{R}^{3\times 3}$ is symmetric positive definite and $e_R:=\mbox{skew}\left(K_{R}R_{e}\right)^{\lor}$ is the left-trivialized derivative of the modified trace function $\Psi_{K_R}$ introduced in \eqref{eq:nav_funct}, where $K_R\in\mathbb{R}^{3\times3}$ is a symmetric matrix satisfying 
\begin{equation}\label{eq:KR}
 tr(K_R)I_3-K_R\in\mathbb{R}^{3\times 3}_{>0}.
\end{equation}
In particular, using $\Psi_{K_R}(R_e)=\frac{1}{2}\mbox{tr}(K_R(I-R_e))$, $e_R\in\mathbb{R}^3$ satisfies \cite{Bullo1999}:
\begin{equation}\label{eq:e_R}
\dot{\Psi}_{K_R}(R_e)=-\frac{1}{2}\mbox{tr}(K_R R_e)=e_R^T R_p e_\omega.
\end{equation}
Using \eqref{eq:Re_dot} and combining the control torque \eqref{eq:cont_torque} with the rotational equations of motion \eqref{dyn4cont2}, \eqref{dyn4cont4}, the dynamics of the errors \eqref{eq:att_err} and \eqref{eq:ang_vel_err} evolves on $\mathrm{SO}(3)\times \mathbb{R}^3$ as:
\begin{align}
\dot{R}_{e}	&=R_{e}{R_p}\hat{e}_{\omega}R_p^T\label{att_err_dyn1}\\
J\dot{e}_{\omega}	&=-R_p^{T}e_{R}-K_{\omega}e_{\omega}-\hat{e}_{\omega}Je_{\omega}-\hat{e}_{\omega}J\omega_p.\label{att_err_dyn2}
\end{align}
The control torque \eqref{eq:cont_torque}, first proposed by \cite{Bullo1999},
has a simpler expression than the one based on the right group error considered in \cite{Leeetal2010} and no cancellation of non-harmful nonlinearities occurs.
The equilibrium points for the attitude subsystem are the points where the differential of $\Psi_{K_R}$ and the angular velocity error vanish, namely:
\begin{equation}
\left\{\begin{array}{cc}
e_R=0\\
e_\omega=0.
\end{array}\right. 
\end{equation}
The set  of equilibria contains the desired equilibrium $(R_e,e_\omega)=(I_3,0)$ and additional undesired configurations corresponding to the other critical points of $\Psi_{K_R}$. This is intrinsic to the structure of $\mathrm{SO}(3)$ and, as a consequence, no time invariant continuous control law can globally stabilize the identity element. Nonetheless, by defining the scalar (positive from \eqref{eq:KR}) :
\begin{equation}
\ell_R:=\lambda_m\left(\mbox{tr}(K_R)I_3-K_R\right)>0,\label{eq:levelset_const}
\end{equation}
it is well known that in the sublevel set 
\begin{align}
\mathcal{S}_R&:=\left\{R_e\in \mathrm{SO}(3):\,\Psi_{K_R}\left(R_e\right)<\ell_R\right\},\label{eq:psi_sublevel}
\end{align}
the point $R_{e}=I_3$ is the unique critical point and minimum of $\Psi_{K_R}$.
Next, the total error energy function
\begin{equation}
V_R\left(R_e,\,e_\omega\right):=\frac{1}{2}e_\omega^T J e_\omega+\Psi_{K_R}\left(R_e\right),\label{eq:V_R}
\end{equation}
will be used in the stability analysis. It can be shown that this function is a Lyapunov candidate for the attitude error dynamics, \emph{i.e.}, it is positive definite about $\left(R_e,\,e_\omega\right)=\left(I_3,\,0\right)$, continuously differentiable and radially unbounded in the direction of $\Vert e_\omega \Vert\rightarrow \infty$. The following theorem,  the proof of which is reported in Section \ref{sec:proof} to avoid breaking the flows of the exposition, establishes desirable properties of the attitude stabilizer.

\begin{theorem}\label{thm:att_asympt}
Consider the system described by \eqref{att_err_dyn1}-\eqref{att_err_dyn2} and a reference attitude  $t\mapsto(R_p(t),\,\omega_p(t))\in T\mathrm{SO}(3)$
satisfying Property~\ref{propty:planner_ref}.
%, such that $\omega_p(t)\in {C}^1\capL_\infty$. 
For any symmetric matrix $K_{R}\in\mathbb{R}^{3\times 3}$ satisfying \eqref{eq:KR} and any symmetric matrix $K_{\omega}\in\mathbb{R}^{3\times 3}_{>0}$, the equilibrium point $\left(R_e,\,e_\omega\right)=\left(I_3,\,0\right)$ is uniformly asymptotically stable with basin of attraction containing the set
\begin{equation}\label{eq:sublevel_att}
 S_a:=\left\{\left(R_e,\,e_\omega\right)\in T\mathrm{SO}(3):\,V_R\left(R_e,\,e_\omega\right)< \ell_R\right\}
\end{equation} 
where $\ell_R$ and $V_R$ are defined in \eqref{eq:levelset_const} and \eqref{eq:V_R}.
\end{theorem}

\begin{remark}
As an alternative choice to \eqref{eq:cont_torque}, the control law
$\tilde{\tau}_{c}  = \hat{\omega} J\omega+J\left(\dot{\omega}_p-\hat{\omega}_p\omega-R_p^T(e_R+K_\omega R_p e_\omega)\right)$
substituted in \eqref{dyn4cont2}, \eqref{dyn4cont4} provides the simpler autonomous closed-loop:
\begin{align}
\dot{R}_e&=R_e \hat{\tilde{e}}_\omega\\
\dot{\tilde{e}}_\omega&=-e_R-K_\omega \tilde{e}_\omega,
\end{align}
where $\tilde{e}_\omega:=R_p(\omega-\omega_p)$. Due to the simpler expression, the attitude planner is independent of the attitude error dynamics and the control scheme is simplified. However, this feedback law cancels non-harmful nonlinearities, at the expense of a larger actuation effort and is therefore less desirable.
\end{remark}

\begin{remark}\label{remark:almost_all}
By choosing $K_R=k_R I_3$, $k_r\in\mathbb{R}_{>0}$, which satisfies \eqref{eq:levelset_const} with $\ell_R=2k_R$, the sublevel set of $\psi_K$ in \eqref{eq:psi_sublevel} contains all the rotations with an angle strictly less than $180^{\circ}$, namely almost all the configurations in $\mathrm{SO}(3)$. Furthermore, by increasing the scalar gain $k_R$, also the set of initial conditions $(R_e(0),e_\omega(0))$  included in \eqref{eq:sublevel_att} can be arbitrarily enlarged. This yields almost-semi global exponential stability \cite{Lee2015c}. Of course, large values for $k_R$ may result in undesired transient performance (overshoot) and a corresponding control torque that cannot be handled by the actuators in practical applications. 
Robust global asymptotic stability of the desired equilibrium $(I_3,0)$ can be obtained by using a hybrid controller on $\mathrm{SO}(3)$ \cite{Mayhew2011}.
\end{remark}

\subsection{Complete dynamics\label{sec:contr}}

This section presents the main results of the stability analysis for the complete system. 
Our proof is based on a compact representation of the closed loop wherein the solutions of the time-varying dynamics is embedded into a time-invariant differential inclusion, in ways that are similar to the strategy in \cite{Naldi2017}, even though the approach adopted here does not require the (somewhat stringent) assumption $\dot \omega_p$ be bounded (this assumption becomes necessary when following the approach in \cite{Naldi2017}). 
%
%For the sake of clarity, we will briefly recall the equations of motion for the feedback interconnection as shown in Figure \ref{fig:cont_scheme}. 
%The state of the overall system, $x:=(R_e,e_\omega,e_x,e_v)$, belongs to the manifold $\mathcal{C}:=T\mathrm{SO}(3)\times\mathbb{R}^6$. 
By introducing $x_a:=(R_e,e_\omega)\in T\mathrm{SO}(3)$ and $x_p:=(e_x,e_x)\in\mathbb{R}^6$,
the solutions to
the attitude error dynamics \eqref{att_err_dyn1}-\eqref{att_err_dyn2} and the position error dynamics 
\eqref{pos_err_dyn1}, \eqref{pos_err_dyn2} can be embedded within the solution funnel generated by the following 
% we can write the complete error dynamics for the attitude (A) and  position (P) dynamics as the 
 constrained differential inclusion:
\begin{align}\label{eq:comp_dyna}
(A)\quad &
\dot{x}_a\in F_a(x_a), && x_a\in T\mathrm{SO}(3)\\
\label{eq:comp_dynb}
(P)\quad & \dot{x}_p \in F_p\left(x_p,x_a\right),  && x_p\in\mathbb{R}^6,
\end{align}
where we used a slight abuse of notation\footnote{
	To be consistent with the formulation, the differential inclusion should be written by exploiting the vectorization, $\text{vec}(\dot{R_e})\in \text{vec}\left(F_R\left(R_e,\,e_\omega\right)\right)$}  and $F_a(x_a)$, $F_p(x_p)$ is defined as
\begin{align}
&F_a(x_a):=\label{eq:F_a}\\
&\underset{\tiny{\begin{array}{c}
		R_p\in\overline{\mbox{co}}\left(\mathrm{SO}\left(3\right)\right)\\
		\left\Vert \omega_p\right\Vert \leq\omega_{M}
		\end{array}}}{\bigcup}\left[\begin{array}{c}
R_{e}R_p\hat{e}_{\omega}R_p^T\\
-J^{-1}\left(R_p^{T}e_{R}+K_{\omega}e_{\omega}+\hat{e}_{\omega}Je_{\omega}+\hat{e}_{\omega}J\omega_p\right)
\end{array}\right],\nonumber\\
&F_p(x_p,x_a):=
\bigmat{e_v\\\frac{1}{m}\left(\beta(e_{x},\, e_{v})+ f_M\gamma(V_a(R_e,e_\omega)) \overline{\mathcal{B}_{3}}\;\right)}
\end{align}
where $\overline{\mbox{co}}(\cdot)$ denotes the closed convex hull, $\overline{\mathcal{B}_{3}}$ denotes the closed unit ball and $\omega_M\in\mathbb{R}_{>0}$ is a constant the existence of which is ensured by Assumption \ref{assum_setpoint}.
Moreover, function $\gamma$ comes from \eqref{eq:vanishing}, and scalar \begin{equation}\label{eq:f_d_max}
f_M:=\sqrt{3}\lambda_2 + f_M^{ss}
\end{equation}
is a bound on the  term $f_d$ arising from substituting \eqref{eq:feed_force} and \eqref{ssbounds} into \eqref{eq:des_force}.

Based on representation \eqref{eq:comp_dyna}, \eqref{eq:comp_dynb}, asymptotic tracking for the complete dynamics can be proven, under Assumption \ref{assum_setpoint}, as stated by the following Theorem.

\begin{theorem}\label{thm:comp_stab}
Consider the closed-loop system described by \eqref{pos_err_dyn1}, \eqref{pos_err_dyn2} and \eqref{att_err_dyn1}, \eqref{att_err_dyn2} controlled by \eqref{eq:cont_force}, \eqref{eq:cont_torque} and  the planner output given by $(R_p,\omega_p)=(R_d,\omega_d)\in T\mathrm{SO(3)}$, where the desired trajectory $t\mapsto (R_d(t),\omega_d(t),x_d(t),v_d(t))$ satisfies Assumption \ref{assum_setpoint}. Then, if $\Phi(\cdot,\cdot,\cdot)$ is selected according to Property \ref{propty:van_pert}, for any symmetric matrix $K_R$ satisfying \eqref{eq:KR}, any $K_\omega\in\mathbb{R}^{3\times 3}_{>0}$,
% and for any desired trajectory $(R_d,\omega_d,x_d,v_d):\mathbb{R}\rightarrow T\mathrm{SO}(3)\times\mathbb{R}^6$ that satisfies Assumption \ref{assum_setpoint}, 
the point
%\begin{align}\label{eq:attractor_comp_stab}
%\left\{x\in\mathcal{C}:
$(R_{e},\,e_\omega,\,e_x,\,e_v)=(I_3,0,0,0)$
%\right\}
%\end{align}
is robustly uniformly asymptotically stable with basin of attraction containing the set $ S_a\times \mathbb{R}^6$, where $ S_a$ is defined in \eqref{eq:sublevel_att}.
\end{theorem}

\begin{proof}
The cascaded interconnection \eqref{eq:comp_dyna}, \eqref{eq:comp_dynb} comprises the upper subsystem  \eqref{eq:comp_dyna}, whose stability
properties (with domain of attraction $ S_a$)
is established in Theorem~\ref{thm:att_asympt}, and the lower subsystem \eqref{eq:comp_dynb}, which is stabilized by the nested saturation feedback \eqref{eq:betadef} proposed in \cite{Naldi2017}. (Local) stability of the cascade follows from standard reduction theorems for differential inclusions (see, e.g., \cite{MaggioreZac17}) whereas attractivity from 
$ S_a\times \mathbb{R}^6$ can be established using the small signal ISS properties of stabilizer  \eqref{eq:betadef} following the same steps as in \cite[Proof of Prop. 4]{Naldi2017}. Finally, stability and attractivity of the point $(R_{e},\,e_\omega,\,e_x,\,e_v)=(I_3,0,0,0)$ for the closed-loop implies $KL$ asymptotic stability from \cite[Thm 7.12]{Goebel2012}, and then also robust $KL$ asymptotic stability from \cite[Thm 7.21]{Goebel2012}.
\end{proof}

\section{Attitude planning for constraints compliance}\label{sec:att_plan}

% Theorem~\ref{thm:comp_stab} and 
The result of the previous section (notably Theorem~\ref{thm:comp_stab}) establishes robust asymptotic stability of the origin for the error dynamics, regardless of the reference  orientation $(R_p,\omega_p)\in T\mathrm{SO(3)}$. Nonetheless, Theorem~\ref{thm:comp_stab} gives no guarantees about the fact that the force $f_c$ requested by the control scheme satisfies the bounds 
characterized in Section~\ref{sec:input_constr}. In this section we propose to select 
$(R_p,\omega_p)$ according to a dynamic attitude planner, 
as represented in Figure~\ref{fig:cont_scheme}. The attitude planner is
in charge of properly changing the desired attitude to prioritize position over orientation tracking, while respecting the input constraints at hand.  
According to Remark \ref{remark:cont_force},
the attitude planner is intended to include a desired decoupling between attitude and position tracking whenever it is possible, \emph{i.e.}, when this would not result in a violation of the actuation constraints.   With respect to the approach of \cite{Franchi2016}, in which the reference attitude is obtained as the solution of an optimization problem, our design allows to naturally compute a differentiable reference that satisfies Property \ref{propty:planner_ref} so that $\tau_c$ in \eqref{eq:cont_torque} is well defined. 

% The idea to introduce a dynamic attitude planner comes from the hint in Remark \ref{remark:cont_force}: by relaxing the attitude control objective, one can always find an orientation $R_p$ such that the control force, which is obtained by rotating the desired control force $f_d$ with the matrix $R_p$, is inside the admissible region around the vertical body axis. Therefore, by properly controlling the reference attitude, one can ensured the compliance with the force constraints and, whenever possible, track the desired attitude motion $R_d(t)$.

% At this point, it is possible to get a picture of the control scheme that we are developing. The tracking error dynamics in the most general form is a feedback interconnection between the attitude and the position error subsystems (Figure \ref{fig:cont_scheme}); the attitude planner is in charge of properly changing the desired attitude in order to prioritize position over orientation tracking. 

\subsection{Attitude planner dynamics}

There is a natural way to express the reference attitude $R_p$ by noticing that the actuation constraints for all the different configurations are satisfied if the control force is delivered along the positive direction of the $b_3$ axis. Following standard strategies for underactuated UAVs \cite{Leeetal2010}, we introduce a smooth matrix function $R_c(f_d,t):\mathbb{R}^3\setminus\{0\}\times \mathbb{R}_{\geq 0}\rightarrow \mathrm{SO}(3)$ defined as: 
\begin{align}
R_{c} (f_d,t) :=  \left[\begin{array}{ccc}
b_{c_1}(f_d,t) & b_{c_2}(f_d,t) & b_{c_3}(f_d)\end{array}\right],\;
b_{c_3} :=  \frac{f_d}{\bigl\Vert f_d\bigr\Vert}, \label{eq:quad_rot_mat}
\end{align}
where $b_{c_1}(f_d,t)$ and $b_{c_2}(f_d,t)$ are any two orthogonal unit vectors such that $R_c$ defines a rotation matrix. A possible selection is:
\begin{align}\label{eq:bc23}
b_{c_1}  :=  b_{c_2}(f_d,t)\times b_{c_3}(f_d),\quad b_{c_2} :=  \frac{b_{c_3}(f_d)\times b_{d}(t)}{\Vert b_{c_3}(f_d)\times b_{d}(t)\Vert}.
\end{align}
The vector $b_d$ defines the desired heading direction of the UAV: 
\begin{equation}\label{eq:bd}
b_{d}(t):=\bigmat{\cos (\psi_d(t)) &\sin(\psi_d(t)) &0}^T,
\end{equation}
where $t\mapsto \psi_d(t) \in \mathbb{R}$ is the desired yaw angle (which may be extracted from a given $R_d(t)$).   We note that $R_c(f_d,t)$ is well defined as long as $f_d\neq \bigmat{0 &0&0}^T$. Based on $R_c$, we select the reference attitude $R_p$, output of the attitude planner, as:
\begin{equation}
R_p:=R_c(f_d,t)R_r,\label{eq:ref_att_decomp}
\end{equation}
where $R_r\in\mathrm{SO}(3)$ is an additional state of the dynamic attitude planner.
Being an element of $\mathrm{SO}(3)$, the differential equation for the relative attitude $R_r$ can be written as:
\begin{equation}
\dot{R}_r=\hat{\omega}_r R_r, \label{eq:rel_att_dyn}
\end{equation}
where $\omega_r\in\mathbb{R}^3$ is the relative angular velocity, with coordinates in the frame $\mathcal{F}_c$, that will be used as an auxiliary input to control the evolution of $R_r$. For instance, $\omega_r$ can be used to track at best the desired attitude $R_d$ by exploiting a Lyapunov-based design, once a suitable potential function of the desired attitude $R_d$ is provided. Then, the time evolution of $R_r$ can be properly modified to satisfy the actuation constraints. In Section \ref{sec:tilt_case}, an example of this approach is reported for the case of conic actuation limitation (item (\emph{b}) in Section \ref{sec:input_constr}).\\
Finally, the attitude planner has to provide a corresponding velocity reference $\omega_p$, satisfying  $\dot{R}_p=R_p \hat{\omega}_p$, which is computed  by first introducing the angular velocity 
\begin{equation}
\omega_c(R_c,f_d):=(R_c^T \dot{R}_c)^{\vee}\label{eq:quad_ang_vel}
\end{equation}
of the frame $\mathcal{F}_C:=\{b_{c_1},\,b_{c_2},\,b_{c_3}\}$, and then using \eqref{eq:ref_att_decomp} and \eqref{eq:rel_att_dyn} to obtain $\dot{R}_p=\dot{R}_cR_r+R_c\dot{R}_r=R_p R_r^T(\hat{\omega}_c+\hat{\omega}_r ) R_r=R_p\hat{\omega}_p$. In particular, the above relation provides  $\hat{\omega}_p=R_r^T(\hat{\omega}_c+\hat{\omega}_r)R_r$, which, using \eqref{eq:veemap}, gives:
\begin{equation}
\omega_{p}  =  R_{r}^{T}\left(\omega_{c}(R_c,f_d)+\omega_{r}\right).\label{eq:planner_ang_vel}
\end{equation}

As the main goal of the attitude planner is to track at best the desired attitude, it is more convenient to rewrite the dynamics \eqref{eq:rel_att_dyn} in terms of the planner attitude error, \emph{i.e.},
\begin{equation}\label{eq:planner_att_err}
R_e^p:=R_p R_d^T.
\end{equation}
Using then $\dot{R}_d=R_d\hat{\omega}_d$ from Assumption \ref{assum_setpoint} and \eqref{eq:planner_ang_vel}, the overall dynamics of the attitude planner becomes:
\begin{align}
\dot{R}_e^p&=R_e^p R_d(\hat{\omega}_p-\hat{\omega}_d)R_d^T\label{planner_err_dyn}\\
\omega_p&=(R_c^T(f_d,t) R_e^p R_d)^T(\omega_c(R_c,f_d)+\omega_r).\label{eq:planner_ang_vel_2}
\end{align}

\begin{prop}\label{prop:goodR_c}
	By selecting the gain $\lambda_2<\inf_{t\geq 0}\vert m(\dot{v}_{d_3}(t)+g)\vert$ in \eqref{eq:betadef}, the rotation matrix $R_c$ by \eqref{eq:quad_rot_mat}-\eqref{eq:bd}, is well defined. 
\begin{proof}
	The proposition can be demonstrated by inspecting the following inequality:
	\begin{align}\label{eq:des_force_positivity}
	\Vert f_d(t)\Vert\geq \vert f_{d_3}(t)\vert	\geq m\vert \dot{v}_{d_3}(t)+g\vert-\vert \beta_3(e_x,e_v)\vert \\
	 \geq \inf_{t\geq 0}\vert m( \dot{v}_{d_3}(t)+g)\vert-\beta_{3_M}>0,
	\end{align}
	which holds thanks to the equivalence $\beta_{3_M}=\lambda_2$ for the definition of 
	\eqref{eq:betadef} coming from \cite{Naldi2017} and the assumption that $\lambda_2<\inf_{t\geq 0}\vert m( \dot{v}_{d_3}(t)+g)\vert$.
	Then, $b_{c_1}$ and $b_{c_2}$ are unit vectors, orthogonal to each other and with $b_{c_3}$ and well defined $\forall t\geq 0$. Indeed, $b_d(t)$ belongs to the horizontal plane $(i_1,\,i_2)$ by definition \eqref{eq:bd} and the third component of $b_{c_3}$ never vanishes from \eqref{eq:des_force_positivity}. Hence, the cross product $b_{c_3}\times b_{d}$ does not vanish either and so $\Vert b_{c_3}\times b_{d}\Vert\neq 0$ $\forall t\geq 0$ in equation \eqref{eq:bc23}. 
\end{proof}
\end{prop}

The following lemma is a useful link between the output $\left(R_p,\,\omega_p\right)$ of the attitude planner, and the reference motion satisfying the properties of Assumption \ref{assum_setpoint}.
Its proof is given in Appendix \ref{proof_quad_feas_ref}.

% Before introducing the control law for the attitude system, we state the following  lemma about the reference attitude provided by the attitude planner.

\begin{lem}\label{lem:feas_ref} \emph{(Feasibility conditions of the planner output).} If the relative angular velocity is bounded and continuously differentiable, \emph{i.e.}, $\omega_r\in(L_\infty\cap {C}^1)$, and the desired angular velocity $\omega_d$ satisfies Assumption \ref{assum_setpoint}, then the reference attitude motion $\left(R_p,\,\omega_p\right)\in TSO(3)$, obtained according to equations \eqref{eq:ref_att_decomp}, \eqref{eq:planner_ang_vel}, is feasible, in the sense that it satisfies Property~\ref{propty:planner_ref}.
%
 % the following conditions hold $\forall t\geq 0$:
	% \begin{align}\label{eq:planner_setpoint}
	% \dot{R}_p(t)&=R_p(t)\hat{\omega}_p(t)\\
	% \omega_{p}(t)&\in {C}^1\cap L_\infty.
	% \end{align}
\end{lem}

\begin{remark}
  The decomposition of the reference attitude in equation \eqref{eq:ref_att_decomp} allows to effectively account for the actuation constraints of the different configurations. The angular velocity $\omega_r$ of the relative rotation matrix is an additional degree of freedom that can be exploited when the thrust-vectoring capability is  not locked, \emph{i.e.}, the control force can be produced in a region around the vertical axis. This additional input can be exploited to track the desired attitude at best while taking into account the constraints. Indeed, it is always possible to select an initial condition $R_r(0)$ such that the actuation constraints are verified. Then, the evolution of the relative attitude can be properly controlled by modifying the angular velocity input $\omega_r$ to satisfy the constraints.
\end{remark}

% \begin{proof}
% 	See Appendix \ref{proof_quad_feas_ref}.
% \end{proof}
\begin{remark}
	The requirement $x_d(\cdot)\in {C}^4$ in Assumption \ref{assum_setpoint} allows us to properly define the time derivative of $\omega_p$, which is required to apply the attitude control law, as shown in the next section. Indeed, by direct computation \eqref{eq:planner_ang_vel}:
	\begin{equation}\label{eq:planner_ang_acc}
	\dot{\omega}_{p}  =  -R_{r}^{T}\hat{\omega}_{r}\omega_{c}+R_{r}^{T}\left(\dot{\omega}_{c}+\dot{\omega}_{r}\right),
	\end{equation}
	where 
	\begin{equation}
	\dot{\omega}_c=(R_c^T \ddot{R}_c-\hat{\omega}_c^2)^{\vee}.\label{eq:quad_ang_acc}
	\end{equation}
\end{remark}
Explicit dependences are not reported in the above equations to avoid an overloaded notation.

\subsection{Special selections}\label{sec:special_sel}

In this section we illustrate the relevance of the proposed control scheme
for addressing some of the input limitations initially discussed in Section~\ref{sec:input_constr}.
To this end, we consider the overall control system comprising control law \eqref{eq:des_force}, \eqref{eq:cont_force} with $\Phi(R_e,e_\omega,t)$ defined according to \eqref{eq:scaling_fact} and satisfying Property \ref{prop:scaling_funct}, the rotation matrix $R_c$ and the angular velocity $\omega_c$ defined in equations \eqref{eq:quad_rot_mat} and \eqref{eq:quad_ang_vel}, respectively, and denote the closed-loop state as
$x:=(R_e,e_\omega,e_x,e_v, R_p^e)$, belonging to the manifold $T\mathrm{SO}(3)\times\mathbb{R}^6 \times \mathrm{SO}(3)$. Next we characterize the closed-loop properties for a few relevant selections of the attitude planner.

\smallskip

\noindent
{\bf (a) Vectored-thrust configuration.}
In this case it is well-known that only a desired  rotation around the vector $m(\dot{v}_d(t)+ge_3)$\footnote{$m(\dot{v}_d(t)+ge_3)$ is the force, resolved in the inertial frame, required to stay on the nominal position trajectory} can be tracked when one wants to guarantee position tracking under the constraint in \eqref{eq:coplanar_const}.  
Indeed, equation \eqref{eq:coplanar_const} requires the control force vector \eqref{eq:cont_force} (resolved in the inertial frame) to be directed along the positive direction of $b_3$, namely, from \eqref{eq:cont_force}, \eqref{eq:scaling_fact}, we must ensure that for some scalar $T\in\mathbb{R}_{>0}$,
\begin{equation}
f_c=c(R_e,t)R_p^T f_d=T e_3.
\end{equation}
This relationship can be written in vector form as:
\begin{equation}\label{eq:coplanar_const_vectform}
c\left(R_e,t\right)\bigmat{b_{p_1}^Tf_d\\b_{p_2}^Tf_d\\b_{p_3}^Tf_d}=\bigmat{0\\0\\T},
\end{equation}
Therefore, it is enough to select $b_{p_3}=\tfrac{f_d}{\Vert f_d\Vert}$ to comply with the constraint, so that \eqref{eq:coplanar_const_vectform}  is verified with $T=c(R_e,t)\Vert f_d\Vert$. By inspecting the decomposition \eqref{eq:ref_att_decomp} and the expression of $R_c$ in \eqref{eq:quad_rot_mat}, this is obtained by selecting planning strategy that guarantees $R_r(t)e_3=e_3$ $\forall t\geq 0$. A possible solution is to initialize $R_r(0)e_3=e_3$ and use $\omega_r=\Omega_r e_3$, where  $\Omega_r \in \mathbb{R}$ may be designed to track a given angle reference $\psi_d(t)$. However, the easiest solution is to select the input $\omega_r(t)=0$ (static planning), for which $R_r(t)=I_3$ $\forall t\geq 0$. This fixes two out of three parameters in the definition of $R_p=R_c$ and the third one can be used to assign a desired rotation around $\tfrac{f_d}{\Vert f_d\Vert}$.  

\begin{definition}\label{def:track_att_quad}
A trajectory $t\mapsto  (R_d(t),\omega_d(t), x_d(t),v_d(t))\in T\mathrm{SO}(3)\times \mathbb{R}^6$ satisfying Assumption \ref{assum_setpoint} is compatible with the position tracking task for the system of case (a) iff  
\begin{equation}
\frac{e_3^T f^{ss}_c(t)}{\Vert f^{ss}_c(t)\Vert}=1, 
\end{equation}
where $f^{ss}_c(t)$ is defined in equation \eqref{eq:feed_force}.
\end{definition}

Basically, the trackability condition of Definition \ref{def:track_att_quad}, \emph{i.e.}, $e_3^Tf_{c}^{ss}= m e_3^T R_d \left(\dot{v}_{d}+g e_3\right)=m\Vert\dot{v}_{d}+g e_3\Vert$, requires the axis  $R_d e_3=:b_{d_3}$ to be directed as the vector $\left(g e_3+\dot{v}_{d}\right)$: since thrust in the inertial frame can be delivered only along $b_3$, this is the only solution compatible with position tracking. On the other hand, the desired heading direction, which is given by $b_{d_1}$, is freely assignable by selecting a desired angle $\psi_d(t)$. Therefore, the tracking problem for case \emph{(a)} is defined on $\mathbb{S}^1\times \mathbb{R}^3$ and can be embedded in $\mathrm{SO}(3)\times \mathbb{R}^3$ via the assignment $\mathbb{S}^1\ni\psi_d\mapsto R_c(f_{c}^{ss}(t),\psi_d(t))\in\mathrm{SO}(3)$, where $R_c(\cdot,\cdot)$ is the map given by \eqref{eq:quad_rot_mat}-\eqref{eq:bd}. 

\begin{cor}\label{cor:vectored-thrust}
Consider the closed-loop system described by \eqref{pos_err_dyn1}, \eqref{pos_err_dyn2}, \eqref{att_err_dyn1}, \eqref{att_err_dyn2} controlled by \eqref{eq:cont_force}, \eqref{eq:cont_torque}, where the planner output is given by $(R_p,\omega_p)=(R_c,\omega_c)$, with $R_c$ and $\omega_c$ selected as in \eqref{eq:quad_rot_mat}-\eqref{eq:bd} and \eqref{eq:quad_ang_vel}, respectively.  Assume that the desired trajectory $t\mapsto (R_d(t),\omega_d(t), x_d(t),v_d(t))$ is trackable according to Definition \ref{def:track_att_quad}, that  $\Phi(\cdot,\cdot,\cdot)$ is selected according to \eqref{eq:scaling_fact} and satisfies Property \ref{propty:van_pert}, and that the gains $k_1,k_2$ and saturation levels $\lambda_1,\lambda_2$ are selected according to \cite[Prop. 1]{Naldi2017} and Proposition \ref{prop:goodR_c}. Then, for any symmetric matrix $K_R$ satisfying \eqref{eq:KR}, any $K_\omega\in\mathbb{R}^{3\times 3}_{>0}$, the control force \eqref{eq:cont_force}  satisfies the actuation constraint \eqref{eq:coplanar_const} and the point $(R_{e},\,e_\omega,\,e_x,\,e_v)=(I_3,0,0,0)$
is robustly uniformly asymptotically stable with domain of attraction containing $S_a\times {\mathbb R}^6$, where $S_a$ is given by \eqref{eq:sublevel_att}. 
\end{cor}

\begin{proof}
Robust asymptotic stability of the curve $(R_e,e_\omega,e_x,e_v)=(I_3,0,0,0)$ follows immediately from Theorem~\ref{thm:comp_stab} combined with Lemma~\ref{lem:feas_ref} and the fact that under Assumption~\ref{assum_setpoint} the solutions of the time-varying dynamics are solutions of the differential inclusions \eqref{eq:comp_dyna}-\eqref{eq:comp_dynb}. Note that by considering $\omega_r(t)=0$, the planner has no dynamics ($R_r(t)=I_3$) and the choice $(R_c,\omega_c)$ satisfies the condition of Theorem \ref{thm:att_asympt} thanks to Lemma~\ref{lem:feas_ref}.  
%Note that, because $(e_x,e_v)\rightarrow (0,0)$, $f_d \rightarrow m(\dot{v}_d(t)+ge_3)$ and $b_{c_3}\rightarrow \tfrac{\dot{v}_d(t)+ge_3}{\Vert \dot{v}_d(t)+ge_3\Vert}=: b_{c_3}^{ss}$. By exploiting the equality  $-b_{c_3}\times (b_{c_3}\times b_{d})= \left(I_3-b_{c_3}b_{c_3}^T \right)b_d$, one can show that $b_{c_1}  :=  b_{c_2}(f_d,t)\times b_{c_3}(f_d)=\frac{b_{c_3}(f_d)\times b_{d}(t)}{\Vert b_{c_3}(f_d)\times b_{d}(t)\Vert}\times b_{c_3}(f_d)=\left(I_3-b_{c_3}b_{c_3}^T \right)\tfrac{b_d(t)}{\Vert b_{c_3}\times b_d(t) \Vert }$ and, because $(e_x,e_v)\rightarrow (0,0)$, $b_{c_1}\rightarrow \left(I_3-b_{c_3}^{ss} (b_{c_3}^{ss})^T \right)\tfrac{b_d(t)}{\Vert b_{c_3}^{ss}\times b_d(t) \Vert }:= b_{c_1}^{ss}$. Finally, since  $R_e:=R R_c^T$, $R_e\rightarrow I_3$ and the chain of equalities $b_1=R e_1= R_e R_c e_1=R_e b_{c_1}$, one concludes that $b_1(t)\rightarrow b_{c_1}^{ss}$.
Finally, the actuation constraint in equation \eqref{eq:coplanar_const} is straightforwardly verified for $R_p=R_c$ because $b_{c_3}:=R_c e_3$ is aligned with $f_d$:
\begin{equation}
f_c=c(R_e,t)R_c^T f_d=c(R_e,t)\Vert f_d\Vert e_3.
\end{equation}
\end{proof}

\begin{remark}
	The actuation constraint \eqref{eq:tilt_const} of case(b) in Section \ref{sec:input_constr} is automatically satisfied by imposing $\omega_r(t)\equiv 0,\, R_r(0)=I_3$, since case \emph{(a)} is the strictest one in terms of actuation constraints. However, this choice severely limits the capabilities of thrust-vectoring platforms.
\end{remark}

\smallskip

\noindent
{\bf (b) Thrust-vectoring configuration.}
When the actuation mechanism allows to change the direction of the thrust, the system has the potential of tracking more complex maneuvers. Within the present design, we make use of the relative rotation matrix $R_r$ introduced in \eqref{eq:ref_att_decomp} to perform attitude maneuvers that are compatible with position tracking. By considering that $R_c$ defined in \eqref{eq:quad_rot_mat} satisfies $R_c e_3=\tfrac{f_d}{\Vert f_d\Vert}$, the control force can be written as: 
\begin{equation}\label{eq:cont_forc_2}
f_c=c(R_e,t)R_r^T R_c^T f_d=c(R_e,t) R_r^T \Vert f_d\Vert e_3=c(R_e,t)\Vert f_d\Vert  R_r^T e_3.
\end{equation} 
By substituting \eqref{eq:cont_forc_2} into the constraint \eqref{eq:tilt_const}, we get:  
\begin{equation}\label{eq:cos_equiv}
\cos(\theta_c)=\frac{c(R_e,t)\Vert f_d\Vert  R_r^T e_3}{c(R_e,t)\Vert f_d\Vert }= e_3^T R_r e_3=e_3^T b_{r_3},
\end{equation} 
which shows that to satisfy \eqref{eq:tilt_const}, it is sufficient to guarantee that $e_3^Tb_{r_3}\geq \cos(\theta_M)$. In this section we will show how the solution proposed in \cite{InvernizziAUTO2018} to compute the relative angular velocity $\omega_r$ of equation \eqref{eq:rel_att_dyn} to account for the conic region constraint \eqref{eq:cont_force} and exploit the thrust-vectoring capabilities of tiltrotor configurations can be applied within the present framework. In particular, we will verify that the planner output obtained by selecting
\begin{equation}
\omega_{r}=b_{r_3}\times\text{Proj}_G\left(\omega_r^d\times b_{r_3}\right)+\left(b_{r_3}^{T}\omega_{r}^{d}\right)b_{r_3}\label{eq:rel_ang_vel_tilt}
\end{equation}
where $\text{Proj}_G:\mathbb{R}^3\rightarrow \mathbb{R}^3$ is a geometric projection operator\footnote{The geometric projection operator is an extension of the smooth projection operator, proposed in \cite{CaiSmoothProj} for systems evolving on $\mathbb{R}^n$, to the case of systems evolving on $\mathbb{S}^n$. See [equation 65]\cite{InveLoveArxiv2017} for the explicit expression.} and
\begin{equation}\label{eq:omega_r_des} 
\omega_{r}^d=R_{r}\omega_{d}-\omega_{c}-R_{r}R_{d}^{T}e_{R}^{p}
\end{equation}
with $e_R^p:=k_d\mbox{skew}(R_e^p)$, satisfies Property \ref{propty:planner_ref}. The projection operator keeps the planar vector $b_{r_3}^\perp:=\smallmat{e_1^T b_{r_3}&e_2^T b_{r_3}}$ inside a circle of radius $\delta:=\sin(\theta_M)$, as shown in Figure \ref{fig:tilt_angle_const}, by acting on the vector field $\omega_r^d$ in the region $\left(\tfrac{\delta}{\sqrt{1+\varepsilon}},\delta\right]$, where $\varepsilon \in(0,1)$ is a user-defined parameter.
In particular, thanks to the projection operator, $e_3^T b_{r_3}(t)\geq \cos(\theta_{M})$ $\forall t\geq 0$ and, by virtue of \eqref{eq:cos_equiv}, also:
\begin{equation}\label{eq:cosine_ineq}
\cos\left(\theta_c(t)\right)\geq\cos\left(\theta_{M}\right)\quad\forall t\geq 0.
\end{equation}
As a consequence, if the relative attitude is initialized such that $b_{r_3}$ is inside a cone defined by $\theta_M$ around $b_{c_3}$, it will never leave it. For instance, it suffices to select $R_r(0)=I_3$.

\begin{lem}\label{lem:rel_ang_vel_tilt}
	The relative angular velocity computed according to equation \eqref{eq:rel_ang_vel_tilt} is such that the planner output $(R_p,\omega_p)\in T\mathrm{SO}(3)$ obtained with \eqref{eq:ref_att_decomp}, \eqref{eq:planner_ang_vel}, satisfies Lemma \ref{lem:feas_ref}.
\end{lem}
\begin{proof}
	By exploiting a smooth Projection operator \cite{CaiSmoothProj}, $\omega_r$ defined in \eqref{eq:rel_ang_vel_tilt} is continuously differentiable and its time derivative can be computed as in \eqref{eq:planner_ang_acc}. Then, we can write the following bound on $\omega_r$:
	\begin{align}
	\Vert\omega_{r}\Vert\leq\Vert\text{Proj}_G(\omega_r^d\times b_{r_3})\Vert+\Vert\omega_{r}^{d}\Vert
	\end{align}
	Notice that the projection operator simply removes the radial component of $\dot{b}_{r_3}^{d}$,  when it is working, hence $\Vert\text{Proj}_G(\omega_r^d\times b_{r_3})\Vert\leq\Vert\omega_{r}^{d}\Vert$ because $b_{r_3}$ is a unit vector. Finally, we can conclude the boundedness of $\omega_r^d$ from the following inequality:
	\begin{align}
	\Vert\omega_{r}^d\Vert\leq\Vert\omega_d\Vert+\Vert\omega_c\Vert+\Vert e_R^p\Vert,
	\end{align}
	in which $\omega_d$ is bounded according to Assumption \ref{assum_setpoint}, $\omega_c$ is bounded as shown in Appendix \ref{proof_quad_feas_ref}, and $e_R^p$ is bounded as well, being the left-trivialized derivative of a function defined on a compact manifold.
\end{proof}

\begin{definition}\label{def:track_att_tilt}
A trajectory $t\mapsto  (R_d(t),\omega_d(t), x_d(t),v_d(t))\in T\mathrm{SO}(3)\times \mathbb{R}^6$ satisfying Assumption \ref{assum_setpoint} is compatible with the position tracking task for the system of case (b) if, given $\delta:=\sin(\theta_M)$, $\varepsilon\in(0,1)$ and $\theta_b:=\arcsin\left(\frac{\delta}{\sqrt{1+\epsilon}}\right)$, there exists $\bar{t}\in\mathbb{R}_{\geq 0}$, such that
		\begin{equation}
		\frac{e_3^T f^{ss}_c(t)}{\Vert f^{ss}_c(t)\Vert}\geq\cos(\theta_b)\quad\forall t\geq \bar{t}, 
		\end{equation}
		where $f^{ss}_c(t)$ is defined in equation \eqref{eq:feed_force}.
\end{definition}

According to Definition \ref{def:track_att_tilt}, the attitude motion is compatible with the conic region constraint \eqref{eq:tilt_const} if the angle between the third desired axis $b_{d_3}$ and the nominal force $R_d f_{c}^{ss}(t)$ is within the cone in which the $\mbox{Proj}_G$ operator is not active.

The next corollary combines the results of Theorem \ref{thm:comp_stab} and Lemma \ref{lem:rel_ang_vel_tilt}. 

\begin{cor}\label{cor:comp_stab}
	Consider the closed-loop system described by \eqref{pos_err_dyn1}, \eqref{pos_err_dyn2}, \eqref{att_err_dyn1}, \eqref{att_err_dyn2}, \eqref{planner_err_dyn} controlled by \eqref{eq:cont_force}, \eqref{eq:cont_torque}, where the planner output  $(R_p,\omega_p)$ is given by \eqref{eq:ref_att_decomp} and \eqref{eq:planner_ang_vel_2}, with $R_c$, $\omega_c$ and $\omega_r$ selected as in \eqref{eq:quad_rot_mat}-\eqref{eq:bd}, \eqref{eq:quad_ang_vel} and \eqref{eq:rel_ang_vel_tilt}-\eqref{eq:omega_r_des}, respectively.  Assume that the desired trajectory $t\mapsto (R_d(t),\omega_d(t), x_d(t),v_d(t))$ is trackable according to Definition \ref{def:track_att_tilt}, that $\Phi(\cdot,\cdot,\cdot)$ is selected according to \eqref{eq:scaling_fact} and satisfies Property \ref{propty:van_pert}, and that the gains $k_1,k_2$ and saturation levels $\lambda_1,\lambda_2$ in \eqref{eq:betadef} are selected according to \cite[Prop. 1]{Naldi2017} and Proposition \ref{prop:goodR_c}. Then, for any symmetric matrix $K_R$ satisfying \eqref{eq:KR}, $K_\omega\in\mathbb{R}^{3\times 3}_{>0}$, $\varepsilon\in(0,1)$ and $k_d\in\mathbb{R}_{>0}$, the control force \eqref{eq:cont_force} satisfies the actuation constraint \eqref{eq:tilt_const} and the point $(R_{e},\,e_\omega,\,e_x,\,e_v,R_e^p)=(I_3,0,0,0,I_3)$ is robustly asymptotically stable with  basin of attraction containing the set $ S_a\times \mathbb{R}^6\times S_{ap}$, where ${S}_{a}$ is given by \eqref{eq:sublevel_att} and ${S}_{ap}:=\left\{R_e^p\in\mathrm{SO}(3):\, e_3^T R_d^T(0) (R_e^p)^T\tfrac{f_d(0)}{\Vert f_d(0) \Vert}\geq \cos(\theta_M)\right\}$, with $f_d$  defined in \eqref{eq:des_force}.
\end{cor} 
\begin{proof}
	The proof is based on Theorem \ref{thm:comp_stab}, Lemma \ref{lem:rel_ang_vel_tilt} and \cite[Thm 3]{InvernizziAUTO2018}. Lemma \ref{lem:rel_ang_vel_tilt}  guarantees that the output planner reference $(R_p,\omega_p)\in T\mathrm{SO}(3)$ computed according to \eqref{eq:ref_att_decomp} and \eqref{eq:rel_ang_vel_tilt} satisfies Property \ref{propty:planner_ref}. Therefore, the assumptions of Theorem \ref{thm:comp_stab} are satisfied and the point $(R_{e},\,e_\omega,\,e_x,\,e_v)=(I_3,0,0,0)$ is UAS. Furthermore, because $R_e^p$ evolves on $\mathrm{SO}(3)$ according to \eqref{planner_err_dyn}, \eqref{eq:planner_ang_vel_2} with $\omega_r$ given by \eqref{eq:rel_ang_vel_tilt} and because the projection operator guarantees that equation \eqref{eq:cosine_ineq} holds $\forall t\geq 0$, the first part of the theorem is demonstrated. Then, following the proof \cite[Thm 3]{InvernizziAUTO2018}, combining the trackability condition in Definition \ref{def:track_att_tilt} and the properties of the projection operator together with the expression of $\omega_r^d$ in \eqref{eq:omega_r_des}, we can conclude the second claim of the theorem. 
\end{proof}

\begin{figure}
\begin{centering}
	\includegraphics[scale=0.25]{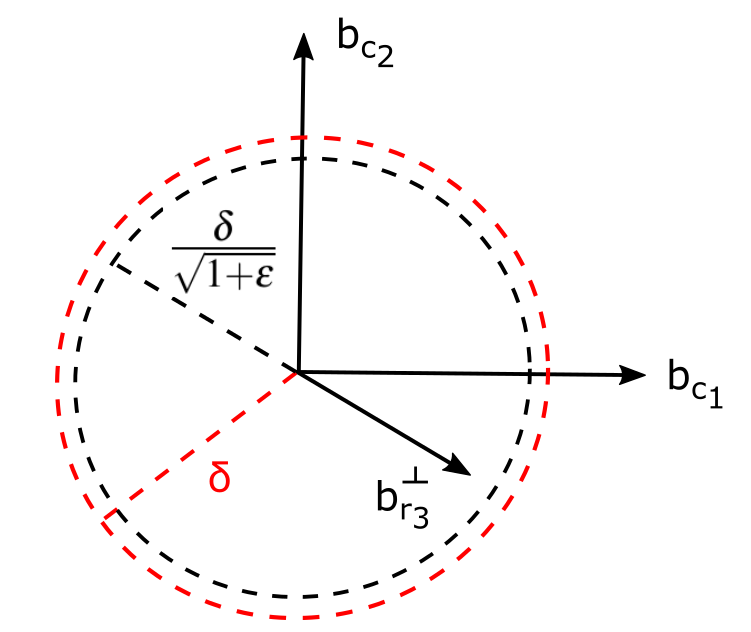}
	\par\end{centering}
\caption{Geometric interpretation of the tilt-angle constraint
	\label{fig:tilt_angle_const}}
\end{figure}

\section{Numerical results: hexacopter with and without tilted propellers} \label{sec:tilt_case}
\subsection{Considered platform}
The control law designed in the previous sections will be applied to the case of a hexacopter UAV with and without tilted propellers. As per the tilted configuration, each propeller is tilted by an angle $\alpha$ with respect to the local $x$-axis, as shown in Figure \ref{fig:Reference-frame-definition}. By setting $\alpha=0^\circ$, the standard coplanar (vectored-thrust) configuration  is recovered (Figure \ref{fig:VT_conf}).
Usually, rotor angular velocities are the physical inputs for multirotor UAVs equipped with DC-brushless propellers, \emph{i.e.}, $u=\smallmat{\omega_1 &\cdots& \omega_n}^T$, as one assumes that fast low level controllers exist to track the desired rotor angular velocities. In this setting, the input constraints are of the kind $\omega_i\in[0,\omega_M]$ $\forall i$, namely, the angular velocity $\omega_i$ of each rotor must be positive (unidirectional propellers) and upper bounded. By means of aerodynamic principles, the propellers deliver the control wrench $w_c=(f_c,\,\tau_c)$ to the UAV, which is related to the physical inputs by following the mixer map:
\begin{align}
f_c&=\sum_{i=1}^6 k_f \omega_i^2 R_{p_i}e_3 \\
\tau_c&=\sum_{i=1}^6\left(\ell_a R_{e_3}\left((i-1)\tfrac{\pi}{3}\right)e_1\times k_f\omega_i^2 R_{p_i}e_3-(-1)^i k_\tau  \omega_i^2 R_{p_i}e_3\right)\nonumber,
\end{align}
where $R_{p_i}=R_{e_3}\left((i-1)\tfrac{\pi}{3}\right)R_{e_1}(\alpha)$, $k_f$ and $k_\tau$ are the thrust and torque coefficients, respectively, whilst $\omega_i$ is the \emph{i-th} rotor angular velocity and $\ell_a$ is the arm length. For a fixed value of the tilt angle $\alpha\in[-90^\circ,90^\circ]$, the mapping can be compactly written as $w_c=M(\alpha) \bar u$, where $\bar  u:=\smallmat{\omega_1^2&\ldots&\omega_6^2}^T\in\mathbb{R}^6$ and $M(\alpha)\in\mathbb{R}^{6\times 6}$ is a constant matrix parametrized by $\alpha$. For the vectored-thrust configuration $(\alpha=0^\circ)$, the rank of $M(\alpha=0)$ is four and it is not possible to compute the physical inputs from $M \bar u=w_c$ unless $w_c$ is in the column space of $M$, \emph{i.e.}, when the only non-null component of the control force $f_c$ is $f_{c_3}$. This is understood as the coplanar platform cannot deliver a control force in the plane perpendicular to $b_3$.  On the other hand, matrix $M(\alpha)$ has rank $6$ whenever $\alpha\neq 0^\circ,\pm 90^\circ,-\arctan(\ell_a k_f/k_\tau), \arctan(k_\tau/(\ell_a k_f))$ (thrust-vectoring configuration) \cite{Morbidi2018}. However, non-feasible, \emph{i.e.}, negative, angular velocities for individual propellers can be obtained by inverting $M$ for a given control wrench. From a physical point of view, this can be understood by considering that the control force delivered according to $M(\alpha)\bar u$ can span only a predefined region in the space around $b_3$, which is dependent on the value of the tilt angle. In particular, the control force must approximately lie within a cone around the third body axis, of an angle dependent on the tilt angle of the propellers. Therefore, the model of case \emph{(b)} in Section \ref{sec:input_constr} is reasonable to approximate the actuation constraint of the considered hexacopter. To ensure the invertibility of the mixer map in a broad range of operating condition, a parameter $\sigma\in (0,1]$ is exploited in the control law presented in Section \ref{sec:special_sel}, such that the conic region \eqref{eq:tilt_const} is $\theta_M=\sigma \alpha$. 

The simulation model used in the following examples is a multi-body system with seven bodies (a central body and six propellers groups), which is written in the Modelica modeling language. The dynamics of the propellers is described by a first order system with time constant $\tau_p=0.05\unit{s}$. Aerodynamic effects are included to increase the reliability of the simulation. Specifically, we consider a damping effect on the rotational dynamics, namely $\tau_d:=-D_a \omega$, where $D_a=\mbox{diag}(0.04,0.04,0.02)$, and the contribution of body and induced drag on the position dynamics, which are collected in the disturbance force \cite{Hua2015}
\begin{equation}\label{eq:drag}
f_d:=-c_d \Vert v \Vert v-\sum_{i=1}^6 c_I \sqrt{T_i}(v_i-(v_i^Tu_i)u_i)
\end{equation}
where $c_d=0.01,\,c_I=0.05$ are the body and induced drag coefficients, respectively, $v_i$ is the velocity of the hub of the \emph{i-th} rotor, $u_i$ is the unit vector describing the current orientation of the \emph{i-th} propeller axis and $T_i$ is the thrust magnitude delivered by the \emph{i-th} rotor.
For the sake of conciseness, we report only the nominal inertial values used for control design. The mass and inertial matrix of the UAV are $ m=1\, \unit{kg}$ and $J=\mbox{diag}\left(0.008,0.008,0.016\right)\unit{kg \cdot m^{2}}$, respectively.
The controller gains are $K_{R}=\mbox{diag}(0.6,0.6,1.4),\, K_{\omega}=0.2 I_3 ,\, \ell=2.1 ,\, k_d=2,\,
\lambda_2=9,\,\lambda_1=1 ,\, k_1=0.06 \quad k_2=9 ,\, \varepsilon=0.05.$
In both the simulations the desired position trajectory is a circle 
\begin{equation}
x_d(t)=\bigmat{\cos(\Omega_d(t) t)& \cos(\Omega_d(t) t) &0}^T\unit{m}
\end{equation}
where $t\mapsto\Omega_{d}(t)$ is made by two constant intervals connected by a smooth ramp (top of Figure \ref{fig:trackability}). The desired attitude is $R_d(t)=I_3$, \emph{i.e.}, the UAV has to fly with level attitude, a requirement that is never compatible with vectored-thrust configurations in the sense of Definition~\ref{def:track_att_quad}. The initial conditions correspond to hover at $x(0)=\smallmat{1&0&0}^T \unit{m}$. 

\subsection{Simulation A - vectored-thrust configuration}

In the first simulation example the proposed control law is applied to the hexacopter with coplanar propellers ($\alpha=0^\circ$).

As shown in Section~\ref{sec:special_sel}, this kind of platform cannot track an arbitrary attitude trajectory when position tracking is the primary objective: only a desired yaw angle can be tracked, which, for the present case, is $\psi_d(t)=0$. The attitude planner is implemented according to Corollary \ref{cor:vectored-thrust}, \emph{i.e.}, the static planning strategy of \cite{Leeetal2010} is employed.
Figure \ref{fig:pos_err_VT} shows that, after a transient phase, the position tracking errors remain bounded: aerodynamic drag prevents their actual convergence to zero.  The oscillations shown at steady state are induced by the combination of aerodynamic drag (which works in the direction of $-v$) and the periodic nature of the circular motion. The attitude tracking performance is shown in Figure \ref{fig:att_err_VT}, where the inclination angle of the body axis $b_3$ with respect to the inertial axis $e_3$, \emph{i.e.},  $\theta_v:=\arccos(b_3^T e_3)$, is plotted (top) together with the yaw angle $\psi$ (bottom). Since the desired attitude is $R_d(t)=I_3$, the corresponding desired values are $\theta_v^d(t)=0^\circ$ and $\psi_d(t)=0^\circ$. While the yaw angle is kept close to zero, the UAV cannot fly at  level attitude: the inclination angle of the vehicle is always greater than zero, reaching a pick during the acceleration phase ($\theta_v(t=21\unit{s})\approx 22^\circ$). As expected, the control law allows to track the closest feasible attitude while guaranteeing position tracking, up to the effect of disturbances and unmodeled dynamics.  

\begin{figure}
	\begin{centering}
		\includegraphics[scale=0.2]{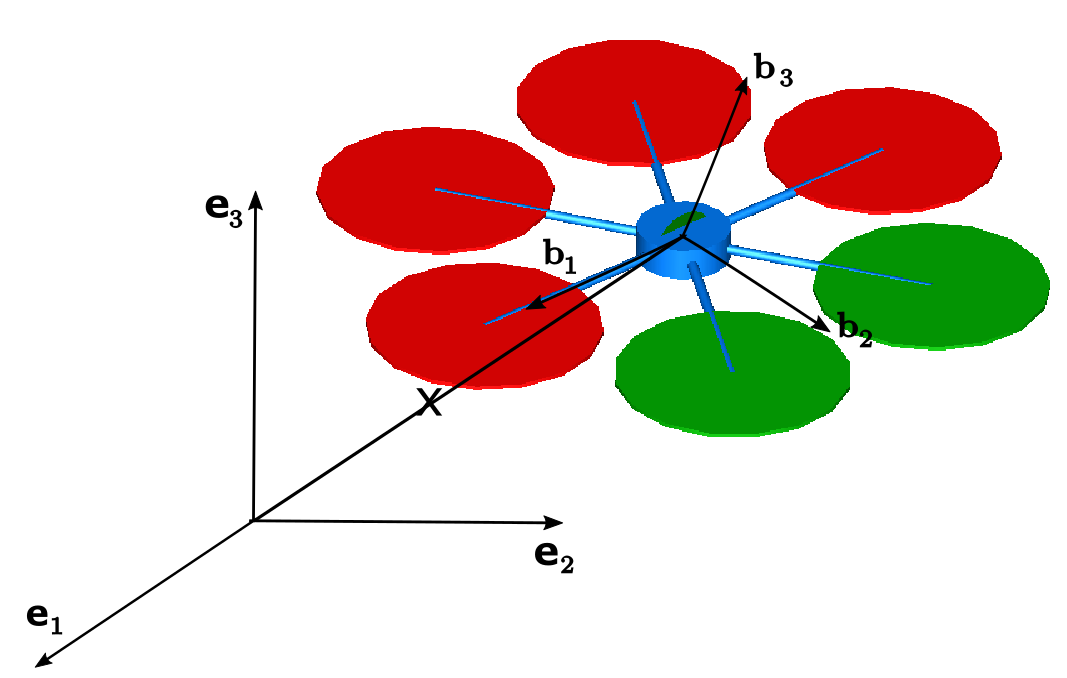}
		\par\end{centering}
	\caption{Hexacopter with vectored-thrust configuration ($\alpha=0^\circ$)\label{fig:VT_conf}}
\end{figure}

\begin{figure}[ht!]
	\begin{centering}
		\includegraphics[scale=0.5]{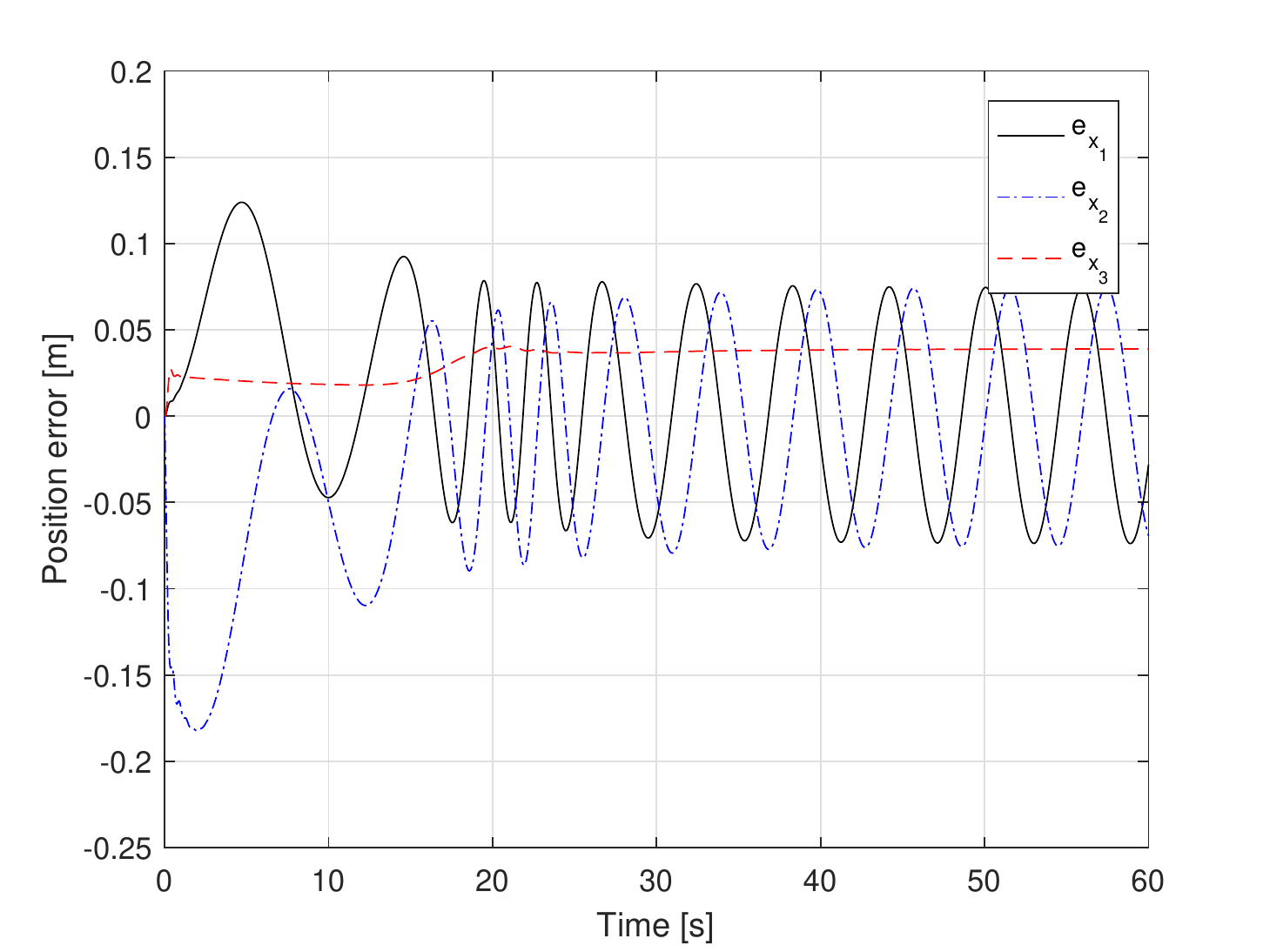}
		\par\end{centering}
	\caption{Position tracking error $e_{x}$.\label{fig:pos_err_VT}}
\end{figure}
\begin{figure}[ht!]
	\begin{centering}
		\includegraphics[scale=0.5]{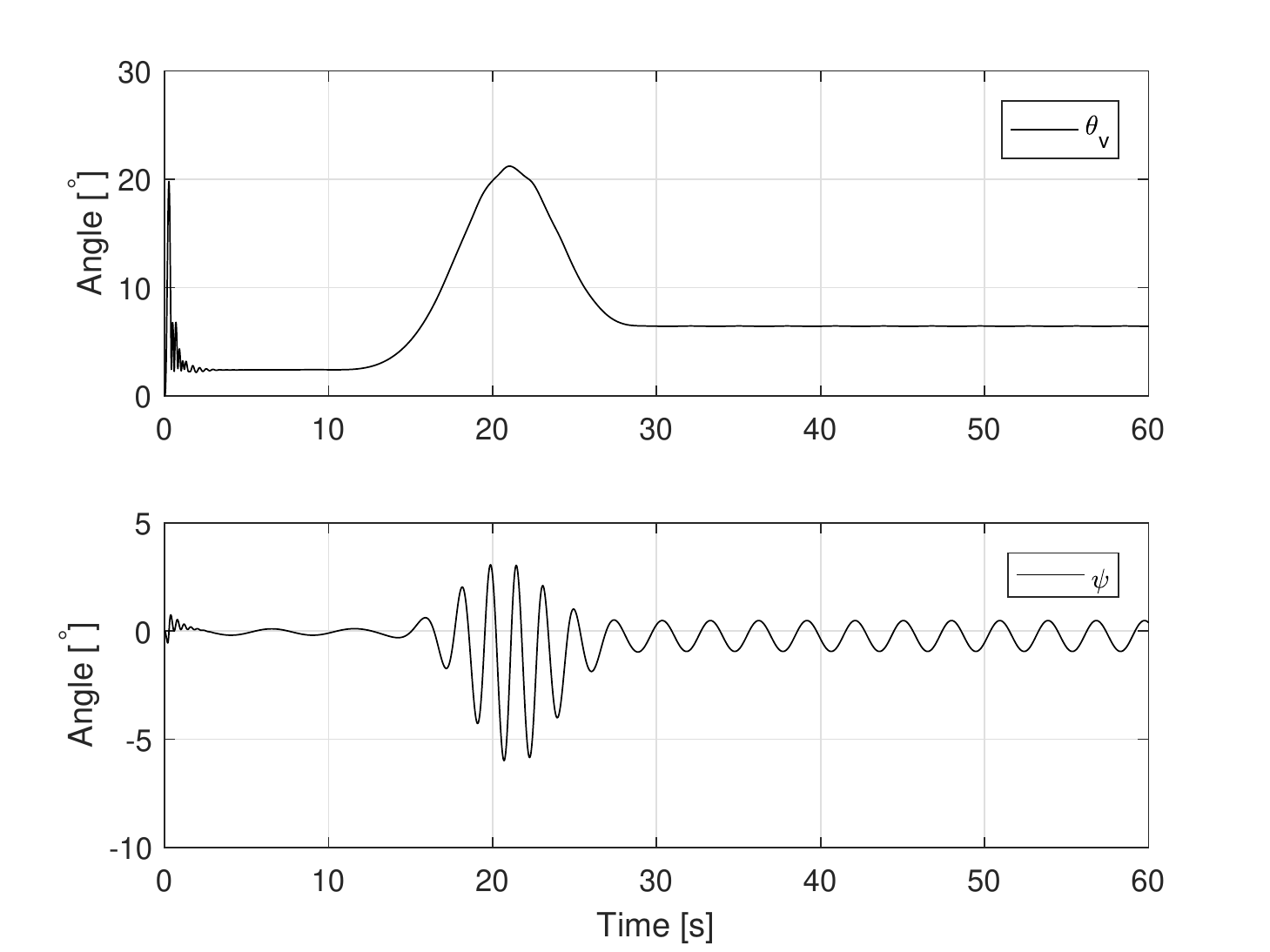}
		\par\end{centering}
	\caption{Attitude tracking - Inclination (top) and yaw angle (bottom).\label{fig:att_err_VT}}
\end{figure}

\subsection{Simulation B - thrust-vectoring configuration}

In the second simulation example we consider the hexacopter in the tilted-propellers configuration. Specifically, each propeller is tilted of an angle $\alpha=20^\circ$. In this case, the conic region \eqref{eq:tilt_const} with $\theta_{M}=10^{\circ}$ was found to be a reasonable approximation of the wrench map $M(\alpha)$ to ensure its invertibility in a broad operational range. 

By inspecting Figure \ref{fig:trackability} (bottom), the considered trajectory is not always feasible in the sense of Definition \ref{def:track_att_tilt}. Indeed, the inclination angle of the steady state control force \eqref{eq:feed_force} with respect to the vertical body axis, \emph{i.e.}, $\theta_n:=\arccos(e_3^T f^{ss}_c(t)/\Vert f^{ss}_c(t))$, is greater than the maximum tilt angle $\theta_M$ in the interval $t\in[17, 26]\unit{s}$. It is worth mentioning that $\theta_n$ represents also the inclination that a vectored-thrust UAV would have to reach to guarantee position tracking (compare Figure \ref{fig:att_err_VT} (top) and Figure \ref{fig:trackability} (bottom)).
Figure \ref{fig:pos_err} shows that, after the transient phase, the position tracking errors remain bounded with steady state oscillations induces by aerodynamic drag. It is interesting to observe that even though the inertial and control parameters are the same for the two simulations, the position errors are slightly smaller for the tilted-propellers UAV (compare Figure \ref{fig:pos_err_VT} and Figure \ref{fig:pos_err}): aerodynamic drag depends upon the system configuration \eqref{eq:drag}, in particular on the UAV attitude. The control force $f_c$ as computed according to equation ~\eqref{eq:cont_forc_2} is reported in Figure~\ref{fig:control_force}. For what concerns the attitude tracking performance, Figure \ref{fig:att_err} depicts the inclination angle $\theta_v$ (top) together with the yaw angle $\psi$ (bottom). When the trajectory is feasible, the UAV is capable of flying at almost  level attitude ($\theta_v(t) \approx 0.5^\circ$). On the contrary, during the initial transient and the acceleration phase, when the desired attitude is not compatible with the thrust-vectoring constraint \eqref{eq:tilt_const} and position tracking, the attitude tracking objective is only partially achieved and the projection operator is working to modify the attitude reference so that position tracking is guaranteed. Note from Figure \ref{fig:trajectory} and Figure \ref{fig:att_err} that the vehicle is inclined of an angle $\theta_v\approx 11.7^\circ$ at $\bar{t}\approx 21\unit{s}$ whereas the nominal angle is $\theta_n (\bar{t})\approx 21^\circ>\theta_M$. Therefore, the proposed solution tries to stay as close as possible to the desired attitude ($\theta^d_v=0$) even if the trajectory is not trackable. Furthermore, the primary objective is not affected during this phase: the position tracking performance is not deteriorated. Finally, the conic region constraint \eqref{eq:tilt_const} is satisfied at all times, \emph{i.e.}, $\theta_c(t) \leq \theta_M$, as  shown in Figure \ref{fig:trackability} (bottom, solid line). The attitude planner reference $\omega_p$ is reported in Figure~\ref{fig:planner_ang_vel}, which confirms the smoothness of the signal computed according to equation \eqref{eq:planner_ang_vel_2} when the planner is modifying the desired attitude to satisfy constraint~\eqref{eq:tilt_const}.
 
\begin{figure}[ht!]
	\begin{centering}
		\includegraphics[scale=0.6]{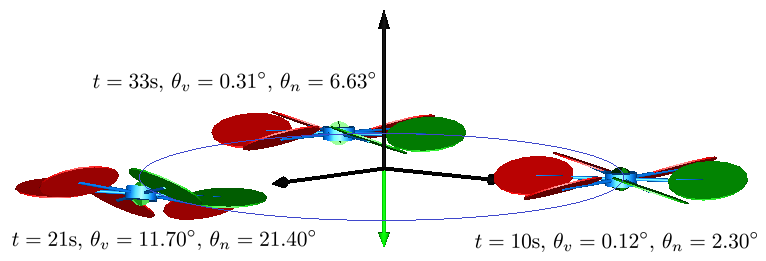}
		\par\end{centering}
	\caption{Trajectory followed by the hexacopter.
		\label{fig:trajectory}}
\end{figure}
\begin{figure}[ht!]
	\begin{centering}
		\includegraphics[scale=0.5]{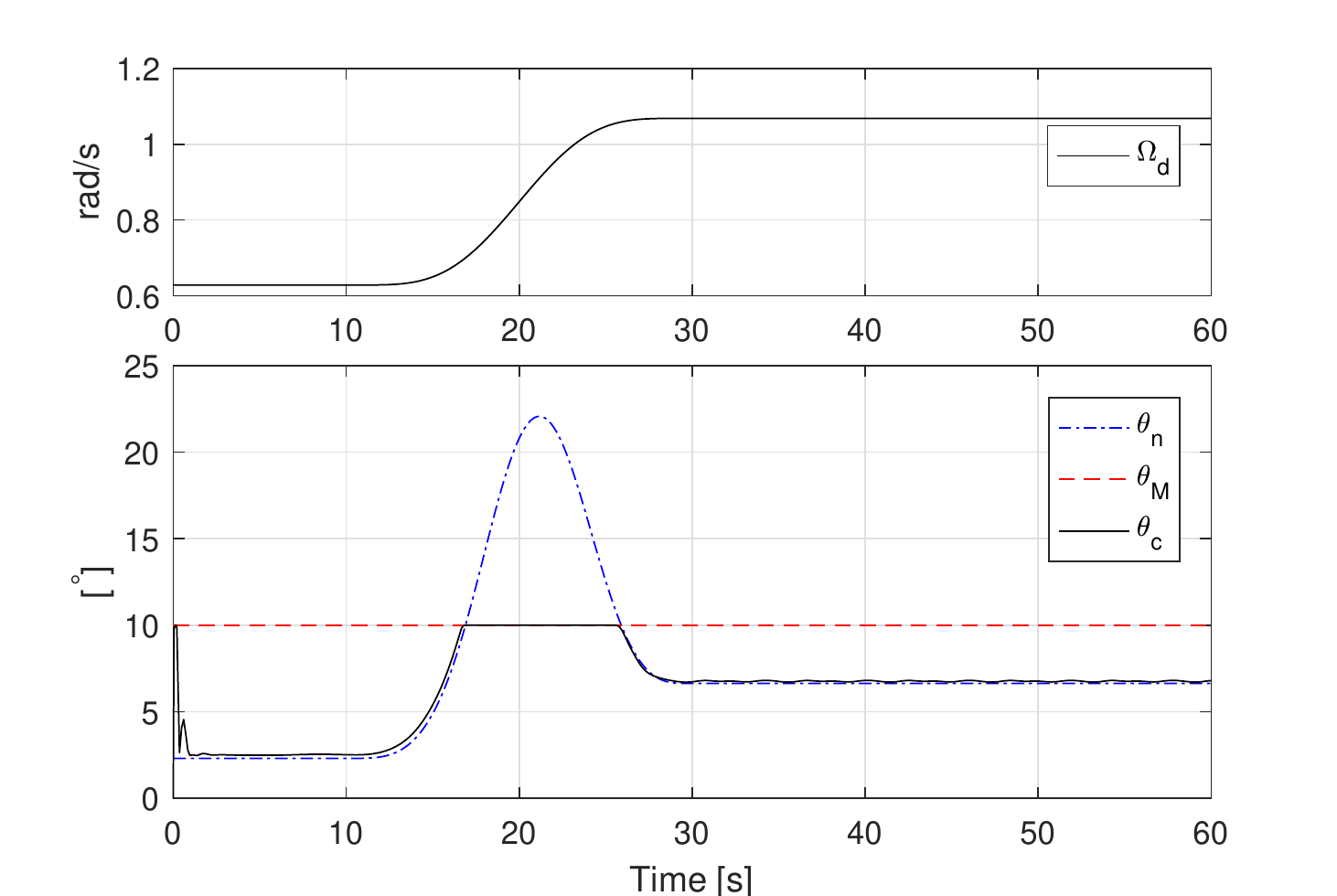}
		\par\end{centering}
	\caption{$\Omega_d(t)$ (top) - Maximum tilt angle $\theta_M$, nominal angle $\theta_n$ of the steady state force with respect $b_3$ and control force angle $\theta_c$ (bottom).\label{fig:trackability}}
\end{figure}
\begin{figure}[ht!]
	\begin{centering}
		\includegraphics[scale=0.5]{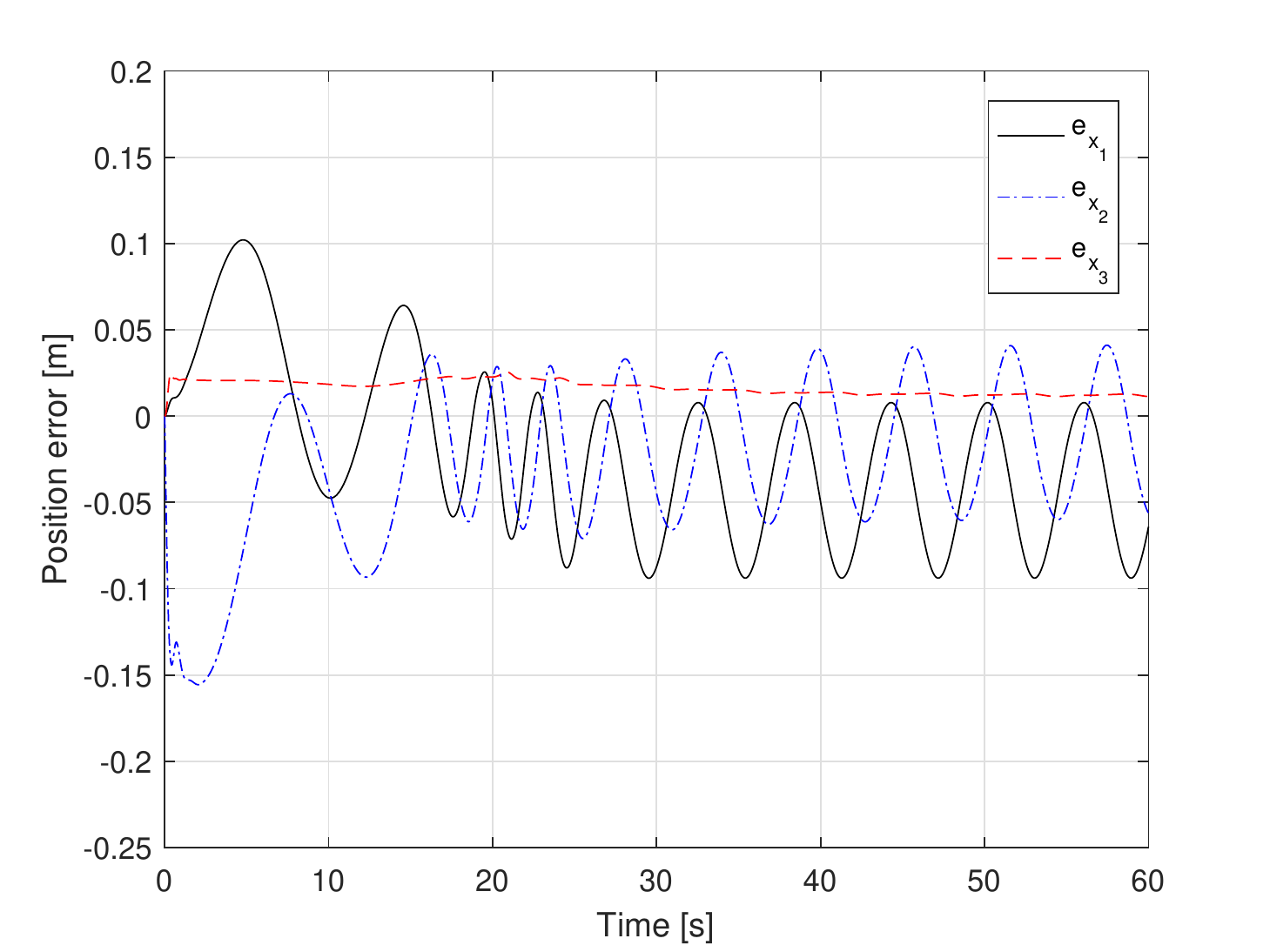}
		\par\end{centering}
	\caption{Position tracking error $e_{x}$.\label{fig:pos_err}}
\end{figure}
\begin{figure}[ht!]
	\begin{centering}
		\includegraphics[scale=0.5]{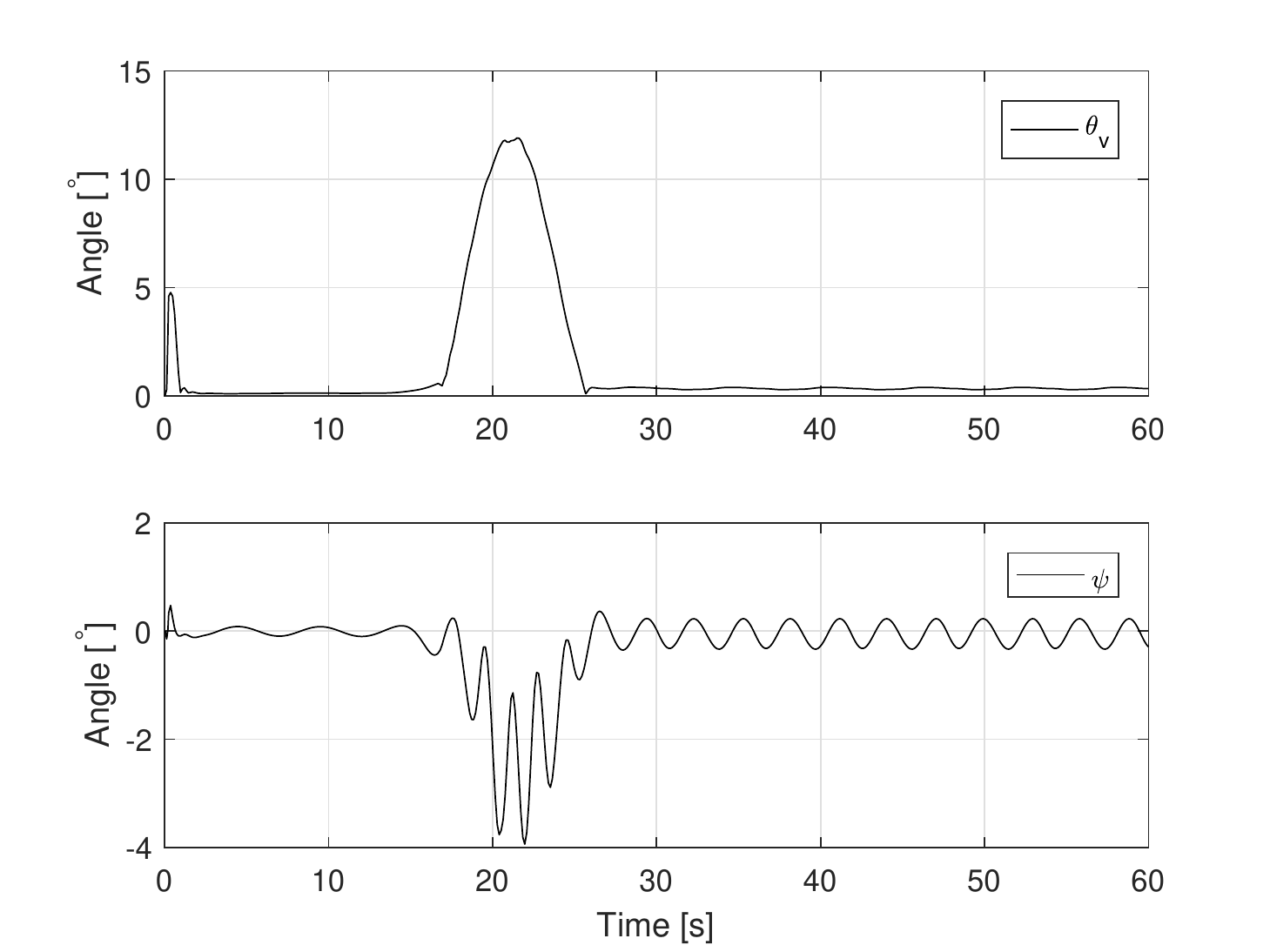}
		\par\end{centering}
	\caption{Attitude tracking - Inclination (top) and yaw angle (bottom).\label{fig:att_err}}
\end{figure}

\begin{figure}[ht!]
	\begin{centering}
		\includegraphics[scale=0.5]{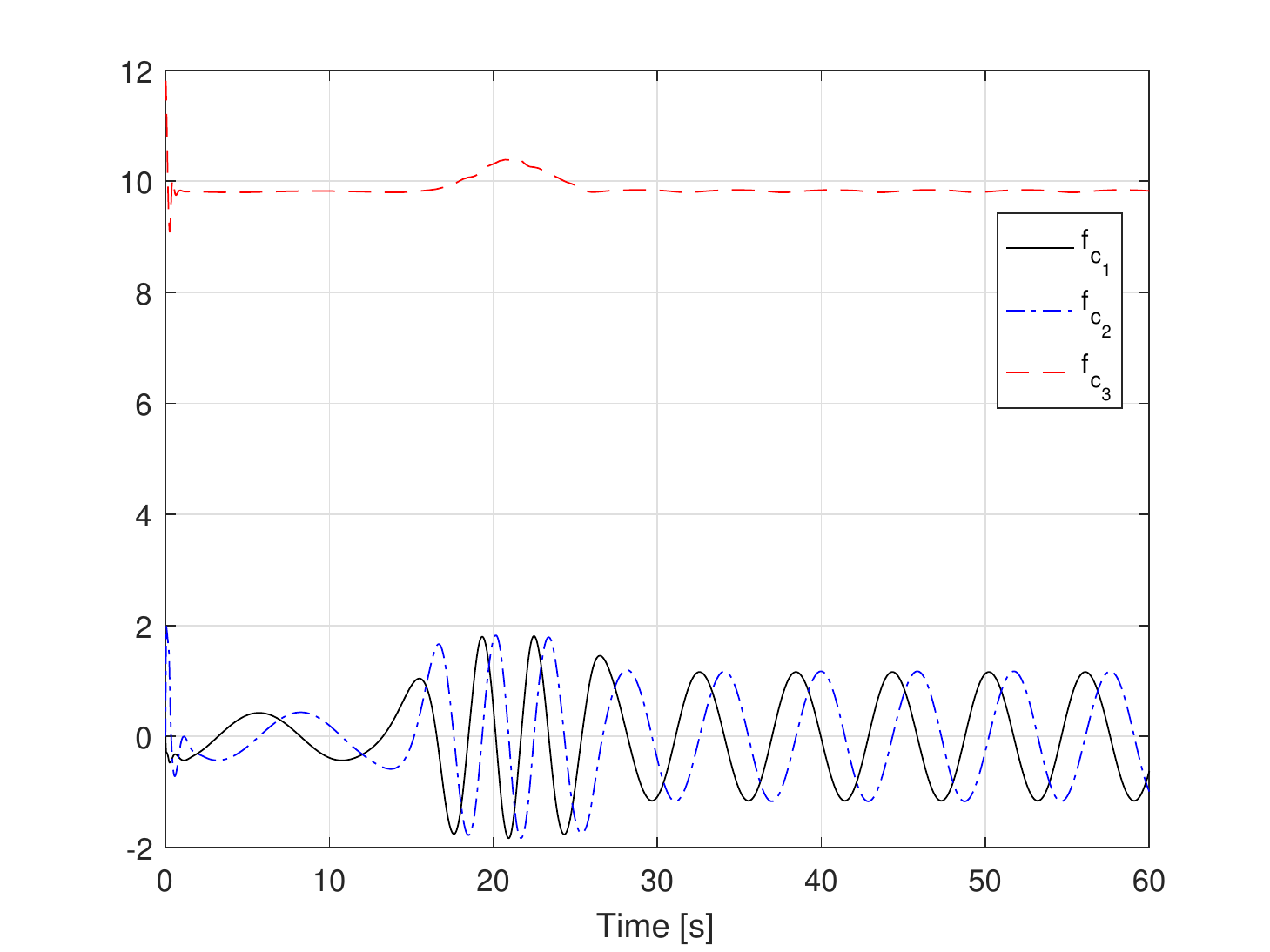}
		\par\end{centering}
	\caption{Control force - $f_c$.\label{fig:control_force}}
\end{figure}

\begin{figure}[ht!]
	\begin{centering}
		\includegraphics[scale=0.5]{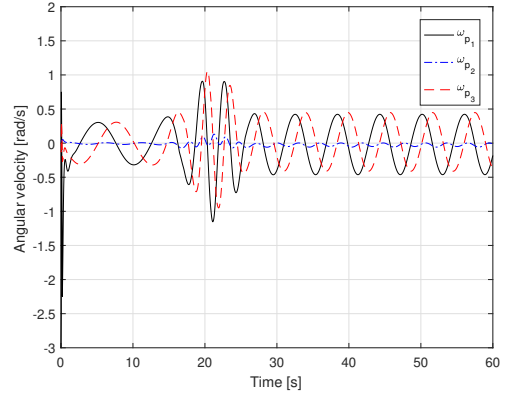}
		\par\end{centering}
	\caption{Angular velocity computed by the attitude planner - $\omega_p$.\label{fig:planner_ang_vel}}
\end{figure}

\section{Proof of the main results}\label{sec:proof}
\subsection{Proof of Theorem \ref{thm:att_asympt}}\label{Proof_att_asympt}
To prove Theorem \ref{thm:att_asympt}, we write the attitude error dynamics as a constrained differential inclusion and then we  apply an invariance principle to prove the asymptotic stability of the desired attractor.  Equations \eqref{att_err_dyn1}, \eqref{att_err_dyn2} describe the evolution of the attitude error dynamics, whose solutions are also solutions of the following differential inclusion
\begin{align}\label{eq:diff_incl}
\left[\begin{array}{c}
\dot{R}_{e}\\
\dot{e}_{\omega}
\end{array}\right]&\in\left[\begin{array}{c}
\begin{array}{c}
F_{R}\left(R_{e},\,e_{\omega}\right)\end{array}\\
F_{\omega}\left(R_{e},\,e_{\omega}\right)
\end{array}\right]=F_a\left(R_{e},\,e_{\omega}\right),
\end{align}
where $F_a$, defined in \eqref{eq:F_a}, is clearly an outer semicontinuous (its graph is closed) and locally bounded set-valued map, and  for each $x_a\in T\mathrm{SO}(3)$, $F_a\left(x_a\right)$ is nonempty and convex.
To prove Theorem \ref{thm:att_asympt} it will be convenient to intersect $F_a\left(x_a\right)$ with the tangent cone to $T\mathrm{SO}(3)$, denoted as $T_{T\mathrm{SO}(3)}(x_a)$ at $x_a$. Indeed the solutions to \eqref{eq:diff_incl} can only flow along the directions in the tangent cone, which simplifies the Lyapunov analysis (see \cite{Seuret2016} for details). To this end, the following lemma is useful:
\begin{lem}\label{flow_intersect}
Given $F_a$ defined in \eqref{eq:F_a} and the closed set $T\mathrm{SO}(3)$, we have the following:
\begin{align}
\label{intersect}
&F_a\left(x_a\right)\cap T_{T\mathrm{SO}(3)}(x_a)=\\
&\quad \underset{\tiny{\begin{array}{c}
R_p\in\mathrm{SO}\left(3\right)\\
\left\Vert \omega_p\right\Vert \leq\omega_M
\end{array}}}{\bigcup}\left[\begin{array}{c}
R_p R_{e}R_p^T\hat{e}_{\omega}\\
-J^{-1}\left(R_p^{T}e_{R}+K_{\omega}e_{\omega}+\hat{e}_{\omega}Je_{\omega}+\hat{e}_{\omega}J\omega_p\right)
\end{array}\right].
\nonumber
\end{align}
\end{lem}
\begin{proofof}{\em Lemma \ref{flow_intersect}.}
For a smooth manifold, the tangent cone is equivalent to the tangent space, namely $T_{T\mathrm{SO}(3)} (x_a)=T_{x_a} T\mathrm{SO}(3)$. Since $T\mathrm{SO}(3)$ is a trivial bundle, \emph{i.e.}, $T\mathrm{SO}(3)=\mathrm{SO}(3)\times\mathbb{R}^3$, then we can write $T_{(R_e,e_\omega)}T\mathrm{SO}(3)\simeq T_{R_e} \mathrm{SO}(3)\times \mathbb{R}^3$. Hence, equality \eqref{intersect} follows from the observation that  $R_{e}{R_{p}}\hat{e}_{\omega}R_p^T\in T_{R_e} \mathrm{SO}(3) \longleftrightarrow R_p\in \mathrm{SO}(3)$.
\end{proofof}

To prove asymptotic stability, we employ the following invariance principle, which is a corollary of \cite[Thm 1]{Seuret2016}.
\begin{prop}\label{propLaSalle}
Consider $\dot{x}\in F(x),\,x\in{\mathcal C}$, where ${\mathcal C}$ is a closed set, $F$ is a set-valued mapping outer semicontinuous and locally bounded relative to $\mathcal{C}$, $F\left(x\right)$ is nonempty and convex $\forall x\in{\mathcal C}$. Given a compact set ${\mathcal A}$, if there exists a continuously differentiable function $V$, positive definite and radially unbounded around ${\mathcal A}$ relative to $\mathcal{C}$ and such that 
\begin{equation}\label{eq:V_dot}
\dot{V}(x)=\underset{\,f\in F(x)\cap T_\mathcal{C}(x)}{\text{max}}\left\langle \nabla V(x),\,f\right\rangle \leq0\quad\forall x\in{\mathcal C}\setminus{\mathcal A},
\end{equation} 
then ${\mathcal A}$ is globally stable, namely Lyapunov and Lagrange stable. Furthermore,
\begin{enumerate}
\item if there exists an open neighborhood $\mathcal{U}\supset\mathcal{A}$ from which no complete solution $\gamma$ exists, satisfying $V(\gamma(t)))=V(\gamma(0))\neq0$, then $\mathcal{A}$ is asymptotically stable.
\item Any such neighborhood $\mathcal{U}$ of the form 
\begin{equation}\label{eq:U_sublevet}
\mathcal{U}:=\left\{x\in\mathcal{C}:\,V(x)<\ell\right\},\,\ell>0,
\end{equation}
is contained in the basin of attraction of $\mathcal{A}$.
\end{enumerate}
\end{prop}

\begin{proofof}{\em Theorem \ref{thm:att_asympt}.}
	
Since all solutions to \eqref{att_err_dyn1}, \eqref{att_err_dyn2} are also solution to \eqref{eq:diff_incl}, we prove the theorem by applying Proposition \ref{propLaSalle} to \eqref{eq:diff_incl} with $\mathcal{A}=\{x_a^*\}=\{I_3,0\}$, with the Lyapunov candidate $V_R$ in \eqref{eq:V_R} which is clearly positive definite and radially unbounded. In particular, for $V=V_R$, we first prove \eqref{eq:V_dot} and then we prove item 2) of Proposition \ref{propLaSalle} with $\mathcal{U}= S_a$, according to \eqref{eq:psi_sublevel} and \eqref{eq:U_sublevet}.
\subsubsection*{Verification of \eqref{eq:V_dot}.}
From Lemma \ref{flow_intersect}, we may evaluate $\dot{V}_R$ in \eqref{eq:V_dot} with $F=F_a$ and $F_a\left(x_a\right)\cap T_{T\mathrm{SO}(3)}\left(x_a\right)$ given in \eqref{eq:F_a}, as follows (where we use $e_\omega^T(\hat{e}_\omega a)=0$ and \eqref{eq:e_R}):
\begin{align}
\dot{V}_R(x_a)&=e_\omega^T\left(J \dot{e}_\omega+R_d^T e_R\right)\nonumber\\
&=-e_\omega^T K_\omega e_\omega\leq 0\quad, \forall x_a \in T\mathrm{SO}(3)\label{eq:V_R_dot}
\end{align}
\subsubsection*{Verification of item 2) of Proposition \ref{propLaSalle}.} According to \eqref{eq:U_sublevet} and \eqref{eq:psi_sublevel}, select $\mathcal{U}=\left\{x_a\in T\mathrm{SO}(3):\,V_R(x_a)<\ell\right\}$ for some $\ell>0$. Thanks to the properties of the potential function $\Psi_K$, the sublevel set of the form $\left\{x_a\in     T\mathrm{SO}(3):\,V_R(x_a)<\ell\right\}$,  where $\ell\leq \ell_R$, with $\ell_R$ defined in \eqref{eq:levelset_const}, contains only the desired equilibrium point. Indeed, $V_R(x_a)<\ell_R$ implies $\Psi_K(R_e)<\ell_R$ for any $x_a\in T\mathrm{SO}(3)$ and only the desired equilibrium point is contained in this sublevel set.
Furthermore, since this set is forward invariant from \eqref{eq:V_R_dot} and the viability condition of   \cite[Pag. 124]{Goebel2012}) is satisfied, solutions starting in $\mathcal{U}$ are complete. We refer now to a solution $t\mapsto \gamma(t)$ starting in $\mathcal{U}$ for which $V_R(\gamma(0))=a\neq0$. Then, if we consider  $\gamma(0)\in\left\{x_a\in\mathcal{U}\,:e_\omega\neq0\right\}$, the function $V_R(\gamma(t))$ has to decrease in time by continuity, which implies $\dot{V}_R(\gamma(t))<0$. Instead, if $\gamma(0)\in\mathcal{G}:=\left\{x_a\in\mathcal{U}:\, e_\omega=0\right\}\setminus \{x_a^*\}$, then, according to the closed-loop dynamics
\begin{equation}
F_a(x_a)\bigr\vert_{\mathcal{G}}=
\underset{\tiny{\begin{array}{c}
R_p\in\overline{\mbox{co}}\left(\mathrm{SO}\left(3\right)\right)\\
\left\Vert \omega_p\right\Vert \leq\omega_M
\end{array}}}{\bigcup}\left[\begin{array}{c}
0\\
-J^{-1}R_{p}^{T}e_{R}
\end{array}\right],
\end{equation}
$\gamma(t)$ will exit the set $\mathcal{G}$ for some small $t>t_0$, since $e_R\neq 0\;\forall x_a\in\mathcal{U}\setminus \{x_a^{*}\}$. As a consequence, $V_R(\gamma(t))$ is forced to decrease again. We can conclude that there is no complete solution that keeps $V_R(\gamma(t))$ constant and different from zero. Therefore, from Proposition \ref{propLaSalle} the proof is complete.
\end{proofof}

\section{Conclusions \label{sec:conclusions}}
In this paper the trajectory tracking control problem for VTOL UAVs with
and without thrust-vectoring capabilities has been addressed. We proposed a priority-oriented control paradigm, in which position tracking is the primary objective. This has been obtained by introducing an attitude planner, in charge of providing a modified attitude reference, which guarantees that the control force, required to track the desired position, can always be delivered by the actuation mechanism. 
Two numerical simulations have been performed by applying the proposed control strategy to a vectored-thrust and to a thrust-vectoring multirotor UAV. The robustness of the control law has been tested by accounting for actuators dynamics, unknown mass distribution, aerodynamic drag and  disturbance torque which were not included in the control design model.  
Future work will be oriented to improve the position tracking performance by considering performance-oriented feedback stabilizers. Indeed, nested saturations, although very robust, are known to be a conservative solution from the performance point of view \cite{Marchand2005a}. For instance, the use of the Quasi Time-Optimal stabilizer for the saturated double integrators, proposed in \cite{Forni2010}, could be considered as an alternative, performance-oriented, solution. 
However, the stability analysis of the closed-loop is not trivial as it is not straightforward to guarantee the $ISS$ property with respect to perturbations. Therefore, this extension may require to employ different mathematical tools to address the cascade analysis.
  
\appendix
\section{Appendix}
\subsection{Proof of Proposition \ref{prop:scaling_funct}}\label{proof_scaling_funct}
First of all, to ease the notation in the following developments, consider the definition,
based on \eqref{eq:att_err}:
\begin{align}
R_p^T R_e R_p&=R_p^T R R_p^T R_p
=R_p^TR:=R_e^r,
\label{eq:righterrordef}
\end{align} 
which is known in the literature as the \emph{right} attitude error \cite{Bullo1999}.
We prove Proposition \ref{prop:scaling_funct} for the general case 
corresponding to fixing any scalar $\psi_M>\max_{R \in SO(3)} \Psi_K(R)$ %(defined in equation \eqref{eq:levelset_const}) 
and selecting
\begin{align}
c(R_e,t):=\frac{\psi_M-\Psi_K(R_p^T R_e R_p)}{\psi_M}=\frac{\psi_M-\Psi_K\left(R_e^r\right)}{\psi_M}=:\bar c\left(R_e^r\right).
\label{eq:cnew}
\end{align}
%where $\psi_M>\ell_R$ (equation \eqref{eq:levelset_const}). Note that 
In particular, the selection of $c(R_e,t)$ in the statement of the proposition corresponds to 
\eqref{eq:cnew} with
%the special case 
 $\psi_M=\ell$ and
%\begin{equation}
%\label{eq:K_scaling}
$K=e_3 e_3^T$.
%\end{equation}
%we recover the definition of $c(\cdot,\cdot)$ in the above proposition.

Note that given $\Psi^2(R_e^r):=\frac{1}{4}\mbox{tr}(I_3-R_e^r)\in[0,1]$ (the normalized distance on $\mathrm{SO}(3)$), the following equality holds:
\begin{equation}
\mbox{tr}(R_e^r)=3-4\Psi^2(R_e^r).
\label{eq:normdist}
\end{equation}
Using \eqref{eq:righterrordef} and \eqref{eq:normdist}, the term $\Delta R$ in equation \eqref{eq:att_mismatch} is bounded by (where we use the cyclic property of the trace function):
\begin{align}
&\Vert \Delta R\Vert \leq \left\Vert \bar c\left(R_e^r\right)R_p R_e^r R_p^T-I_3\right\Vert
\leq \left\Vert \bar c\left(R_e^r\right)R_p R_e^r R_p^T -I_3\right\Vert _{F}\nonumber\\
&=\sqrt{\mbox{tr}\left(\left(\bar c\left(R_e^r\right)R_p R_e^r R_p^T-I_3\right)^{T}\left(\bar c\left(R_e^r\right)R_p R_e^rR_p^T-I_3\right)\right)}\nonumber\\
&=\sqrt{\mbox{tr}\left(R_p \left( \bar c\left(R_e^r\right)^{2}I_3-\bar c\left(R_e^r\right)\left(R_e^r+\left(R_e^r\right)^{T}\right)+I_3\right) R_p^T \right)}\nonumber\\
&=\sqrt{3\left(1+\bar c\left(R_e^r\right)^{2}\right)-\bar c\left(R_e^r\right)\mbox{tr}\left(R_e^r+\left(R_e^r\right)^{T}\right)}\nonumber\\
&=\sqrt{3\left(1+\bar c\left(R_e^r\right)^{2}\right)-2\bar c\left(R_e^r\right)\mbox{tr}\left(R_e^r\right)}\nonumber\\
&= \sqrt{3\left(1+\bar c\left(R_e^r\right)^{2}\right)-2\bar c\left(R_e^r\right)\left(3-4\Psi^2\left(R_e^r\right)\right)}\nonumber\\
&=\sqrt{3\left(\bar c\left(R_e^r\right)-1\right)^2+8\bar c\left(R_e^r\right)\Psi^2\left(R_e^r\right)}.
\end{align}
Therefore, by substituting $0<\bar c\left(R_e^r\right)\leq 1$ from \eqref{eq:cnew},
\begin{align}
\Vert \Delta R\Vert&\leq\sqrt{3 \tfrac{\Psi_K^2(R_e^r)}{\psi_{M}^2}+8\Psi^2(R_e^r)}\leq\sqrt{\tfrac{12+8\psi_M^2}{\psi_{M}^2}}\Psi(R_e^r) =:\varrho\Psi(R_e^r)
\end{align}
where we have used the inequalities $\Psi_K(R_e^r) \leq 2\Psi^2(R_e^r)$ and $\Psi^4(R_e^r)\leq\Psi^2(R_e^r)\; \forall R_e^r\in\mathrm{SO}(3)$ (by definition $\Psi(R_e^r)\in[0,1]$). 
% Then, for $K$ given in \eqref{eq:K_scaling}, $\lambda_M(\mbox{tr}(K)I_3-K)=1$, so that $\varrho:=\sqrt{\tfrac{3+8\psi_M^2}{\psi_{M}^2}}$. Finally, since $\Psi(R_e^r)=\Psi(R_e)$ (the normalized distance is the same for the left and right representation thanks to the properties of the trace operator) and $\Psi(R_e)\in[0,1]$, there exists a bounded class$-\mathcal{H}_I$\footnote{A function $\gamma:\mathbb{R}_{\geq 0}\rightarrow \mathbb{R}_{\geq 0}$ is of class$-\mathcal{H}_I$ if it is of class$-\mathcal{K}$ and satisfies $\int_{0}^{1}\frac{\gamma(s)}{s} ds<\infty$.} function $\gamma(\cdot)$  such that:
% \begin{equation}
% \Vert \Delta R\Vert= \varrho \Psi(R_e)\leq \gamma\left(\sqrt{\Vert e_\omega\Vert^2+\Psi^2(R_e)}\right)=:\gamma(V_a)
% \end{equation}

\subsection{Proof of Lemma \ref{lem:feas_ref}}\label{proof_quad_feas_ref}
The proofs that $\dot{R}_p=R_p\hat{\omega}_p$ and that $\omega_{p}(\cdot)\in C^1$ are an immediate consequence of the definitions given in equations \eqref{eq:quad_ang_vel}, \eqref{eq:quad_ang_acc} and the assumed smoothness properties of the relative angular velocity $\omega_r$. By definition of $\omega_p$ in \eqref{eq:planner_ang_vel}, one gets the following inequality:
\begin{equation}
	\Vert \omega_p \Vert = \Vert R_r^T(\omega_c+\omega_r)\Vert=\Vert \omega_c+\omega_r\Vert\leq \Vert \omega_c\Vert +\Vert \omega_r\Vert
\end{equation}
Hence, $\omega_p$ is bounded as long as $\omega_c$ is bounded because $\omega_r\in L_\infty$ by assumption. To prove the boundedness of $\omega_c$, we first note that:
\begin{equation}\label{eq:omega_c_mod}
\Vert \omega_c \Vert=\frac{1}{2}\Vert b_{c_1}\times\dot{b}_{c_1}+b_{c_2}\times\dot{b}_{c_2}+b_{c_3}\times\dot{b}_{c_3}\Vert,
\end{equation}
by resorting to Poisson's formula. Thus, the bound on $\omega_c$ is:
\begin{equation}
\Vert \omega_c \Vert\leq \frac{1}{2}\left(\Vert \dot{b}_{c_1} \Vert+\Vert \dot{b}_{c_2} \Vert+\Vert \dot{b}_{c_3} \Vert\right).
\end{equation}
The time derivatives of the unit vectors are:
\begin{align}
\dot{b}_{c_1}&=  \dot{b}_{c_2}\times b_{c_3}+ b_{c_2}\times \dot{b}_{d_3}\nonumber\\
\dot{b}_{c_2} & =  \frac{\dot{b}_{c_3}\times b_{d}+b_{c_3}\times \dot{b}_{d}}{\Vert b_{
c_3}\times b_{d}\Vert}\\
&\quad -\left(b_{c_3}\times b_{d}\right)^T	\left(\dot{b}_{c_3}\times b_{d}+b_{c_3}\times \dot{b}_{d}\right)\frac{b_{c_3}\times b_{d}}{\Vert b_{c_3}\times b_{d}\Vert^3} \nonumber\\
\dot{b}_{c_3} & =  \frac{\dot{f}_{d}}{\Vert f_{d}\Vert}-\left(\dot{f}_{d}^{T}f_{d}\right)\frac{f_{d}}{\Vert f_{d}\Vert^3}.
\end{align}
By exploiting the triangular inequality, one gets:
\begin{align}
\Vert \dot{b}_{c_1} \Vert & \leq  \Vert\dot{b}_{c_2}\Vert+\Vert \dot{b}_{c_3}\Vert\label{eq:bc1_bound}\\
\Vert \dot{b}_{c_2} \Vert  & \leq 2 \frac{\Vert \dot{b}_{c_3}\Vert +\Vert\dot{b}_{d}\Vert}{\Vert b_{c_3}\times b_{d}\Vert}\label{eq:bc2_bound}\\
\Vert \dot{b}_{c_3} \Vert  & \leq 2 \frac{\Vert\dot{f}_{d}\Vert}{\Vert f_{d}\Vert}.\label{eq:bc3_bound}
\end{align}
First of all, note that the denominators $\Vert b_{c_3}\times b_{d}\Vert$ and $\Vert f_d \Vert$ are well defined  (bounded and different from zero) $\forall t\geq 0$ according to Remark \ref{prop:goodR_c} and the choice of the reference heading direction $b_d$ in \eqref{eq:bd}. 
Then, the boundedness of $\{\dot{b}_{c_1},\dot{b}_{c_2},\dot{b}_{c_3}\}$ holds since $\Vert f_{d}\Vert\leq f_M$ according to \eqref{eq:f_d_max} and $\Vert \dot{f}_{d}\Vert\leq \dot{f}_M$ thanks to the feedback stabilizer properties $\beta(e_x,\,e_v)$ selected as in \eqref{eq:betadef}. 
For what concerns  $\dot{f}_d$, by differentiating equation \eqref{eq:des_force}, we get $
\dot{f}_d =\dot{\beta}(e_x,e_v)+m\ddot{v}_d =\nabla_{e_x}\beta(e_x,e_v)e_v+\nabla_{e_v}\beta(e_x,e_v)\dot{e}_v+m\ddot{v}_d$,
from which the following inequality is derived:
\begin{align}\label{eq:fd_dot}
\Vert\dot{f}_d\Vert \leq\Vert \nabla_{e_x}\beta(e_x,e_v)\Vert \Vert e_v\Vert+\Vert\nabla_{e_v}\beta(e_x,e_v)\Vert \Vert\dot{e}_v\Vert+m\Vert\ddot{v}_d\Vert.
\end{align}
By recalling equations \eqref{pos_err_dyn2}, \eqref{eq:f_d_max} and Property \ref{propty:van_pert}, the acceleration error is bounded by:
\begin{align}
m\Vert\dot{e}_{v}\Vert	&\leq \Vert\beta(e_x,\,e_v)\Vert+  \Vert \Delta R\left(R_{e},e_{w},t\right)\Vert \Vert f_d\Vert\nonumber\\
&\leq \beta_M+\gamma(V_a)f_M=\sqrt{3}\lambda_2+\gamma(V_a)f_M.
\end{align}
Since the gradient $\nabla\beta(e_x,e_v)$ of the nested saturation stabilizer in \eqref{eq:betadef} is bounded and $\nabla_{e_x}\beta(e_x,e_v)$ vanishes outside the set $\varepsilon_{e_M}:=\left\{(e_x,e_v)\in\mathbb{R}^6:\vert e_{v_i}\vert\geq \lambda_1+\tfrac{\lambda_2}{k_2},\,i=1,2,3\right\}$ (by suitably selecting the saturation function in \eqref{eq:betadef}), $\Vert \dot{f}_d\Vert$ in \eqref{eq:fd_dot} is bounded as long as $m\Vert \ddot{v}_d\Vert$ is bounded (Assumption \ref{assum_setpoint}). Hence, $\dot{b}_{c_3}$ is bounded and, by inspecting \eqref{eq:bc2_bound}, also $\dot{b}_{c_2}\in   L_\infty$,  because $\dot{b}_d\in L_\infty$ as it is dependent on $\omega_d$ (which is bounded according to Assumption \ref{assum_setpoint}). By combining these results, we get that also $\dot{b}_{c_1}\in L_\infty$ according to \eqref{eq:bc1_bound}. Finally, by referring to \eqref{eq:omega_c_mod}, $\omega_c$ is bounded as well and, therefore, $\omega_p \in L_\infty$.

\bibliographystyle{plain}
\bibliography{thesis}

\begin{thebibliography}{10}

\bibitem{Mahony}
{ R. Mahony, V. Kumar and P. Corke }.
\newblock {Multirotor Aerial Vehicles: Modeling, Estimation and Control of
  Quadrotor}.
\newblock {\em IEEE Robotics \& Automation Magazine}, 19(3):20--32, 2012.

\bibitem{Bullo1999}
F.~Bullo and Murray R.~M.
\newblock Tracking for fully actuated mechanical systems: a geometric
  framework.
\newblock {\em Automatica}, 35(1):17--34, jan 1999.

\bibitem{CaiSmoothProj}
Z.~Cai, M.~S. de~Queiroz, and D.~M. Dawson.
\newblock A sufficiently smooth projection operator.
\newblock {\em IEEE Transactions on Automatic Control}, 51(1):135--139, Jan
  2006.

\bibitem{Crowther2011}
B.~Crowther, A.~Lanzon, M.~Maya-Gonzalez, and D.~Langkamp.
\newblock Kinematic analysis and control design for a nonplanar multirotor
  vehicle.
\newblock {\em Journal of Guidance, Control, and Dynamics}, 34(4):1157--1171,
  2011.

\bibitem{Forni2010}
F.~Forni, S.~Galeani, and L.~Zaccarian.
\newblock A family of global stabilizers for quasi-optimal control of planar
  linear saturated systems.
\newblock {\em IEEE Transactions on Automatic Control}, 55(5):1175--1180, May
  2010.

\bibitem{Franchi2016}
A.~Franchi, R.~Carli, D.~Bicego, and M.~Ryll.
\newblock Full-pose tracking control for aerial robotic systems with laterally
  bounded input force.
\newblock {\em IEEE Transactions on Robotics}, PP(99):1--8, 2018.

\bibitem{Goebel2012}
R.~Goebel, R.~G. Sanfelice, and A.~R. Teel.
\newblock {\em Hybrid Dynamical Systems}.
\newblock University Press Group Ltd, 2012.

\bibitem{Hua2015}
M.D. Hua, T.~Hamel, P.~Morin, and C.~Samson.
\newblock Control of {VTOL} vehicles with thrust-tilting augmentation.
\newblock {\em Automatica}, 52(2):1--7, feb 2015.

\bibitem{InvernizziCCTA18}
D.~Invernizzi, M.~Giurato, P.~Gattazzo, and M.~Lovera.
\newblock Full pose tracking for a tilt-arm quadrotor {UAV}.
\newblock In {\em IEEE Conference on Control Technology and Applications,
  Copenhagen, Denmark}, aug 2018.

\bibitem{InveLoveArxiv2017}
D.~Invernizzi and M.~Lovera.
\newblock Geometric trajectory tracking control of thurst vectoring {UAVs}.
\newblock {\em Submitted to Automatica. Preliminary version arXiv:1703.06443},
  2017.

\bibitem{InvernizziAUTO2018}
D.~Invernizzi and M.~Lovera.
\newblock Trajectory tracking control of thrust-vectoring {UAVs}.
\newblock {\em Automatica}, 95:180 -- 186, 2018.

\bibitem{InvernizziACC18}
D.~Invernizzi, M.~Lovera, and L.~Zaccarian.
\newblock Geometric trajectory tracking with attitude planner for
  vectored-thrust {VTOL UAVs}.
\newblock In {\em 2018 Annual American Control Conference (ACC)}, pages
  3609--3614, 2018.

\bibitem{Jiang2014}
G.~Jiang and R.~Voyles.
\newblock A nonparallel hexrotor {UAV} with faster response to disturbances for
  precision position keeping.
\newblock In {\em 2014 IEEE International Symposium on Safety, Security, and
  Rescue Robotics}, pages 1--5. IEEE, 2014.

\bibitem{koditschek1989}
D.~E. Koditschek.
\newblock {The Application of Total Energy as a Lyapunov Function for
  Mechanical Control Systems}.
\newblock {\em J. E. Marsden, P. S. Krishnaprasad and J. C. Simo (Eds) Dynamics
  and Control of Multi Body Systems}, 97:131--157, February 1989.

\bibitem{Lee2015c}
T.~Lee.
\newblock Global exponential attitude tracking controls on $\mathrm{SO}(3)$.
\newblock {\em IEEE Transactions on Automatic Control}, 60(10):2837--2842,
  2015.

\bibitem{Leeetal2010}
T.~Lee, M.~Leok, and H.~McClamroch.
\newblock {Geometric tracking control of a quadrotor UAV on SE(3)}.
\newblock In {\em IEEE Conference on Decision and Control, Atlanta, USA}, 2010.

\bibitem{MaggioreZac17}
M.~Maggiore, M.~Sassano, and L.~Zaccarian.
\newblock {Reduction Theorems for Hybrid Dynamical Systems}.
\newblock {\em arXiv:1712.03450}, 2017.

\bibitem{Marchand2005a}
N.~Marchand and A.~Hably.
\newblock Global stabilization of multiple integrators with bounded controls.
\newblock {\em Automatica}, 41(12):2147--2152, 2005.

\bibitem{Mayhew2011}
C.~G. Mayhew, R.~G. Sanfelice, and A.~R. Teel.
\newblock Quaternion-based hybrid control for robust global attitude tracking.
\newblock {\em IEEE Transactions on Automatic Control}, 56(11):2555--2566, Nov
  2011.

\bibitem{Michieletto2017b}
G.~Michieletto, A.~Cenedese, L.~Zaccarian, and A.~Franchi.
\newblock Nonlinear control of multi-rotor aerial vehicles based on the
  zero-moment direction.
\newblock In {\em IFAC 20\textsuperscript{th} World Congress, Toulouse,
  France}, 2017.

\bibitem{Morbidi2018}
F.~{Morbidi}, D.~{Bicego}, M.~{Ryll}, and A.~{Franchi}.
\newblock Energy-efficient trajectory generation for a hexarotor with dual-
  tilting propellers.
\newblock In {\em Proc. IEEE/RSJ Int. Conf. Intelligent Robots and Systems
  (IROS)}, pages 6226--6232, October 2018.

\bibitem{Naldi2017}
R.~Naldi, M.~Furci, R.~G. Sanfelice, and L.~Marconi.
\newblock Robust global trajectory tracking for underactuated {VTOL} aerial
  vehicles using inner-outer loop control paradigms.
\newblock {\em IEEE Transactions on Automatic Control}, 62(1):97--112, 2017.

\bibitem{Oosedo2015}
A.~Oosedo, S.~Abiko, S.~Narasaki, A.~Kuno, A.~Konno, and M.~Uchiyama.
\newblock Flight control systems of a quad tilt rotor unmanned aerial vehicle
  for a large attitude change.
\newblock In {\em 2015 IEEE International Conference on Robotics and
  Automation}, pages 2326--2331. IEEE, 2015.

\bibitem{Rajappa2015}
S.~Rajappa, M.~Ryll, H.~H. B{\"u}lthoff, and A.~Franchi.
\newblock Modeling, control and design optimization for a fully-actuated
  hexarotor aerial vehicle with tilted propellers.
\newblock In {\em 2015 IEEE International Conference on Robotics and
  Automation}, pages 4006--4013. IEEE, 2015.

\bibitem{Romero2007}
H.~Romero, S.~Salazar, A.~Sanchez, and R.~Lozano.
\newblock A new uav configuration having eight rotors: Dynamical model and
  real-time control.
\newblock {\em 46th IEEE Conference on Decision and Control}, pages 6418--6423,
  Dec 2007.

\bibitem{Roza2014}
A.~Roza and M.~Maggiore.
\newblock A class of position controllers for underactuated {VTOL} vehicles.
\newblock {\em IEEE Transactions on Automatic Control}, 59(9):2580--2585, 2014.

\bibitem{Ryll2015}
M.~Ryll, H.~H. B{\"{u}}lthoff, and P.~{Robuffo Giordano}.
\newblock A novel overactuated quadrotor unmanned aerial vehicle: Modeling,
  control, and experimental validation.
\newblock {\em IEEE Transactions on Control Systems Technology},
  23(2):540--556, March 2015.

\bibitem{Seuret2016}
A.~Seuret, C.~Prieur, S.~Tarbouriech, A.~R. Teel, and L.~Zaccarian.
\newblock {A nonsmooth hybrid invariance principle applied to robust
  event-triggered design}.
\newblock {\em Rapport LAAS n. 17179}, May 2017.

\bibitem{Zhao2013}
S.~Zhao, W.~Dong, and J.~A. Farrell.
\newblock Quaternion-based trajectory tracking control of {VTOL}-uavs using
  command filtered backstepping.
\newblock In {\em Proc. American Control Conf}, pages 1018--1023, June 2013.

\end{thebibliography}

\end{document}